\documentclass[final,a4paper, 12pt]{article}

\usepackage[a4paper,scale=0.768]{geometry}
\usepackage[english]{babel}

\usepackage[mathscr]{euscript}
\usepackage{color}
\usepackage{fancyhdr}
\usepackage{dsfont}
\usepackage{bm}
\usepackage{hyperref}
\usepackage{cite}
\usepackage{showkeys}
\usepackage{subfigure}
\usepackage{float}

\usepackage{amsfonts}
\usepackage{amsmath}
\usepackage{amssymb}
\usepackage{amscd}
\usepackage{amsthm}

\usepackage{dblaccnt}
\usepackage{accents}
\usepackage{graphicx}
\usepackage{array}

\newtheorem{theorem}{Theorem}
\newtheorem{algorithm}[theorem]{Algorithm}

\newtheorem{corollary}[theorem]{Corollary}

\newtheorem{example}[theorem]{Example}
\newtheorem{remark}[theorem]{Remark}
\newtheorem{assumption}[theorem]{Assumption}

\newcommand*{\N}{\ensuremath{\mathbb{N}}}
\newcommand*{\Z}{\ensuremath{\mathbb{Z}}}
\newcommand*{\R}{\ensuremath{\mathbb{R}}}
\newcommand*{\C}{\ensuremath{\mathbb{C}}}
\renewcommand{\i}{\mathrm{i}}
\renewcommand{\phi}{\varphi}
\renewcommand{\rho}{{\varrho}}
\renewcommand{\epsilon}{{\varepsilon}}

\renewcommand{\d}[1]{\,\mathrm{d}#1 \,}

\newcommand{\J}{\mathcal{J}} 
\newcommand{\0}{{0}} 

\newcommand{\grad}{\nabla}

\newcommand{\W}{{W_{\hspace*{-1pt}{\Lambda}}}} 
\newcommand{\Wast}{{W_{\hspace*{-1pt}{\Lambda}^\ast}}} 

\newlength{\dhatheight}

\usepackage{color}

\begin{document}

\sloppy

\title{Numerical Methods for Quasi-Periodic Incident Fields Scattered by  Locally Perturbed Periodic Surfaces}
\author{Ruming Zhang\thanks{Center for Industrial Mathematics, University of Bremen
; \texttt{rzhang@uni-bremen.de}}}
\date{}
\maketitle

\begin{abstract}
Waves scattering from unbounded structures are always complicated problems for numerical simulations. For the case that the non-periodic incident field scattered by (locally perturbed) periodic surfaces, with the help of the Bloch transform, the problem could be solved by some finite element methods, if the incident fields decay at certain rate at the infinity. For faster decaying incident fields, a high order numerical method is also available. However, in these cases, the plain waves, which belong to a very important family of incident fields but do not decay at the infinity, are not included. In this paper, we aim to develop the Bloch transform based standard finite method for this certain case, and then establish the high order method afterwards. Numerical experiments have been carried out for both the standard and high order numerical methods. Based on the algorithms for incident plain waves, we could also extend the numerical methods to more generalized cases when only the not so efficient standard method is available.
\end{abstract}

\section{Introduction}

In this paper, we will propose a standard and a high order numerical method for the incident plane waves scattered by locally perturbed periodic surfaces. When the surface is not perturbed, as the incident plane wave is quasi-periodic,  the scattered/total field is also quasi-periodic. Thus the scattering problem is reduced to the quasi-periodic scattering problem in one periodic cell, then be solved in a bounded domain. The quasi-periodic scattering problems have been well studied in the past 30 years, see \cite{Kirsc1993,Kirsc1995a} for Dirichlet boundary conditions. However, when the surface is perturbed, the scattered field is no longer quasi-periodic, the problem could not be reduced into one periodic cell, thus it is an unbounded problem, that is always difficult for numerical solutions. Another possibility to solve the problems numerically is to treat the locally perturbed periodic surfaces as rough surfaces, i.e., as was shown in \cite{Chand1996,Chand1996a}, and then solved by the numerical method \cite{Meier2000}. However, in this case, the periodicity is ignored. If we do not want to neglect the periodicity, which might be a great advantage to the numerical methods, the Bloch transform is another possible approach.

The Bloch transform has been used in the electrical engineering approach known as the array scanning method, see \cite{Munk1979,Valer2008}. However, rigorous theoretical analysis has only been considered recently.  In \cite{Coatl2012}, the author applied the Bloch transform to the scattering problems from locally perturbed periodic mediums. The properties of the Bloch transform and the transformed problems have been discussed in \cite{Lechl2015e,Lechl2016}, and numerical methods have been proposed in \cite{Lechl2016a,Lechl2016b,Lechl2017}  for scattering problems from locally perturbed periodic surfaces for both 2D and 3D cases. Based on a further study of the singularity of the Bloch transformed fields in \cite{Zhang2017d}, a high order method has also been established in \cite{Zhang2017e}. However, for all of these methods, the incident fields are always required to be decaying at a certain rate, thus a very important class of incident fields, i.e., the plain waves, are not included. In this paper, based on the known results and methods, we will focus on the numerical solution for the scattering of plane waves, and extend the high order method to more generalized incident fields, compared to \cite{Zhang2017e}.

Due to the non-decaying of the incident field, the total field is not, at least, proved to be decaying at the infinity. This is the key problem to the Bloch transform based methods. However, it is possible to construct a decaying field by subtracting the total field from the non-perturbed surface. As the domains are different for the two cases, we will transform the problem in the perturbed periodic structure into one defined in the non-perturbed one. The difference between the two total fields satisfies the equation with a right hand side with compact support. Applying the techniques in \cite{Lechl2017}, the standard numerical method could be developed. Furthermore, following the analysis  in \cite{Zhang2017d}, we can obtain similar regularity result for the Bloch transformed total field. Based on that, a high order  numerical method for the new problem is also introduced. We have also worked on the convergence and error analysis of the numerical methods. Different from \cite{Lechl2017,Zhang2017e}, when the total field  from the periodic surface is approximated by piecewise linear functions, the term in the right hand side is a $H^1$-function, which will be more difficult for the error estimation. However, it could be improved if a $C^1$ finite element space is used. 

Besides the scattering problems of the plane waves, we can also apply the technique in this paper to more generalized problems. For the cases that the incident fields scattered by periodic surfaces could be solved numerically, we can always develop a high order method, no matter if the conditions in \cite{Zhang2017e} is satisfied. The new method, which involves a step to the numerical solution of the scattering from periodic surfaces and a step the high order method, will be much more efficient, compared to the old one in \cite{Lechl2017}.

The rest of the paper is organized as follows. In the second section, the well-posedness of the scattering problems from rough surfaces will be introduced, and in the third section, we will give a description of the quasi-periodic scattering problems. In Section 4, the new problem of the difference of two total fields will be discussed, and in Section 5, we will study the solvability of the Bloch transformed new problem. In Section 6, we will show the standard finite element method, and the high order method will be introduced in Section 7. In Section 8, we will discuss how to extend the method to more generalized problem. In Section 9, several numerical results will be listed.

\section{Scattering from rough surfaces}

We will consider the scattering problems in the space $\R^2$. 
Let $\zeta$ be a Lipschitz continuous bounded function defined on the real line. Then we can define the surface $\Gamma$ by the graph of function $\zeta$, i.e.,
\begin{equation*}
\Gamma:=\left\{(x_1,\zeta(x_1)):\,x_1\in\R\right\}.
\end{equation*}
The surface is assumed to be sound soft, thus the total field satisfies  Dirichlet boundary condition.

\begin{remark}
In this paper, only the sound-soft surface is considered, as an example. In fact, similar arguments and results could be made for impedance boundary conditions, by extending the result for locally perturbed periodic cases in \cite{Lechl2016}. We can also extend the results to the scattering problems from perturbed rough layers, under certain conditions such that the problems are uniquely solvable, e.g. in \cite{Chand1998,Hu2015}.
\end{remark}

\begin{remark}
\textbf{Only} in this section, $\zeta$ is not required to be periodic.
\end{remark}

Then we can define the domain above $\Gamma$ by
\begin{equation*}
\Omega:=\left\{(x_1,x_2):\,x_2>\zeta(x_1),\,x_1\in\R\right\}.
\end{equation*}
As $\zeta$ is a bounded function, there is a positive constant $H$ such that $H>\|\zeta\|_\infty$. For simplicity, we assume that $\zeta(x_1)>0$ for all $x_1\in\R$. Define the straight line above $\Gamma$ by
\begin{equation*}
\Gamma_H:=\left\{(x_1,H):\,x_1\in\R\right\},
\end{equation*}
and the domain between $\Gamma$ and $\Gamma_H$ by
\begin{equation*}
\Omega_H:=\left\{(x_1,x_2):\,\zeta(x_1)<x_2<H,\,x_1\in\R\right\}.
\end{equation*}

Given an incident field $u^i$ that satisfies the Helmholtz equation 
\begin{equation}\label{eq:sca1}
\Delta u^i+k^2 u^i=0\text{ in }\Omega,
\end{equation}
 then the total field $u$ satisfies Helmholtz equation in $\Omega$ as well, and also satisfied the homogeneous Dirichlet boundary condition on the sound soft boundary $\Gamma$, i.e.,
\begin{equation}\label{eq:sca2}
\Delta u+k^2 u=0\,\text{  in }\Omega,\quad u=0\,\text{ on }\Gamma.
\end{equation}
To guarantee that the scattered field $u^s:=u-u^i$ is propagating upwards, it  satisfies the upward propagating radiation condition (UPRC) above $\Gamma_H$, i.e.,
\begin{equation*}
u^s(x_1,x_2)=\frac{1}{2\pi}\int_\R e^{\i x_1 \xi+\i\sqrt{k^2-|{\xi}|^2}(x_2-H)}\,\hat{u}^s({\xi},H)\d{\xi},\quad { x}_2\geq H,
\end{equation*}
where $\hat{u}^s$ is the Fourier transform of $u^s(\cdot,H)$ and 
\begin{equation*}
\sqrt{k^2-|\xi|^2}=\begin{cases}
\sqrt{k^2-|\xi|^2},\quad \text{ when }|\xi|\leq k,\\
\i\sqrt{|\xi|^2-k^2},\quad \text{ when }|\xi|>k.
\end{cases}
\end{equation*}
 The UPRC is equivalent to the boundary condition, i.e.,
\begin{equation*}
\frac{\partial u^s}{\partial x_2}(x_1,H)=T^+\big[u^s|_{\Gamma_H}\big],\quad x_1\in\Gamma_H,
\end{equation*}
where the Dirichlet-to-Neumann (DtN) map $T^+$ is defined by
\begin{equation}\label{eq:DtN}
T^+\phi=\frac{\i}{{2\pi}}\int_\R \sqrt{k^2-|{\xi}|^s}e^{\i x_1{\xi}}\hat{\phi}({\xi})\d{\xi},
\end{equation}
where $\hat{\phi}(\xi)$ is the Fourier transform of $\phi$ on the real line.  
As was proved in  \cite{Chand2010}, the operator $T^+$ is a continuous operator from $H_r^{1/2}(\Gamma_H)$ into $H_r^{-1/2}(\Gamma_H)$ for any $|r|<1$.
Thus the total field $u$ satisfies 
\begin{equation}\label{eq:boundary_condition}
\frac{\partial u}{\partial x_2}(x_1,H)-T^+\left[u|_{\Gamma_H}\right]=f,\quad \text{ where }f=\frac{\partial u^i}{\partial x_2}(x_1,H)-T^+\left[u^i|_{\Gamma_H}\right],
\end{equation}
which turns the scattering problem into a problem that defined on  the domain $\Omega_H$. 
Then the scattering problem has the variation formulation, i.e., given the incident field $u^i\in H^1(\Omega_H)$, to find a solution $u\in\widetilde{H}^1(\Omega_H)$ such that
\begin{equation}\label{eq:var_origional}
\int_{\Omega_H}\left[\nabla u\cdot\nabla\overline{v}-k^2u\overline{v}\right]\d x-\int_{\Gamma_H}T^+\left[u\big|_{\Gamma_H}\right]\overline{v}\d s=\int_{\Gamma_H}f\overline{v}\d s,
\end{equation}
for all $v\in\widetilde{H}^1(\Omega_H^p)$ with compact support in $\overline{\Omega_H}$. The variational problem could also be analyzed in the weighted Sobolev space $\widetilde{H}_r^1(\Omega_H)$, when $u^i\in H_r^1(\Omega_H)$ for some $|r|<1$.
\begin{remark}
The tilde in $\widetilde{H}_r^1(\Omega_H)$  shows that the functions in this space belong to $H_r^1(\Omega_H)$ and satisfy homogeneous Dirichlet boundary condition on $\Gamma$. Similarly for spaces like $H_0^r(\Wast;\widetilde{H}_\alpha^s(\Omega^\Lambda_H))$.
\end{remark}

In \cite{Chand2010}, the well-posedness of the scattering problem has been proved for certain incident fields and surfaces.

\begin{theorem}\label{th:uni_solv}
For any Lipschitz continuous and bounded function $\zeta$, when the incident field $u^i\in H_r^1(\Omega_H)$ where $|r|<1$, the variational problem \eqref{eq:var_origional} is uniquely solvable in $\widetilde{H}_r^1(\Omega_H)$.
\end{theorem}

\begin{remark}
If the incident field is (quasi-)periodic, for example, is a plane wave, it belongs to the space $H_r^1(\Omega_H)$ for any $-1<r<-1/2$, thus in this case, the well-posedness still holds for any $-1<r<-1/2$. Especially, the cases that the incident fields are plain waves are included.
\end{remark}

In this paper, we will consider the quasi-periodic incident plane waves scattered by locally perturbed periodic surfaces. As was shown in \cite{Lechl2017}, if the Bloch transform is applied directly to the scattering problems, the standard finite element method is only proved when $r>0$. Thus for the case that $r<-1/2$, the finite element method is not, at least, proved to be convergent. So we have to find another way to deal with this kind of problems. As the quasi-periodic incident fields scattered by periodic surfaces are well-studied both theoretically and numerically, it is possible to study the difference of the scattered field from one quasi-periodic incident field with periodic and locally perturbed surfaces. In the next section, we will give a brief introduction to the quasi-periodic scattering problems, and after that, the difference of the two scattered field will be studied. 





\section{Quasi-periodic scattering problems}
From now on, the function $\zeta$ is assumed to be ${\Lambda}$-periodic, where $\Lambda$ is a positive number. Let $\Lambda^*=2\pi/\Lambda$, then define one periodic cell $\W$ and its dual periodic cell $\Wast$ (also called the Brillouin zone) by
\begin{equation*}
\W=\left(-\frac{\Lambda}{2},\frac{\Lambda}{2}\right];\quad\Wast=\left(-\frac{\Lambda^*}{2},\frac{\Lambda}{2}\right]=\left(-\frac{\pi}{\Lambda},\frac{\pi}{\Lambda}\right].
\end{equation*}
Let $\Gamma^\Lambda$, $\Omega^\Lambda$, $\Omega^\Lambda_H$, $\Gamma^\Lambda_H$ be $\Gamma$, $\Omega$, $\Omega_H$, $\Gamma_H$ restricted in one periodic cell $\W\times\R$, i.e.,
\begin{eqnarray*}
&&\Gamma^\Lambda=\Gamma\cap\left[\W\times\R\right],\quad \Gamma^\Lambda_H=\Gamma_H\cap\left[\W\times\R\right];\\
&&\Omega^\Lambda=\Omega\cap\left[\W\times\R\right],\quad \Omega^\Lambda_H=\Omega_H\cap\left[\W\times\R\right].
\end{eqnarray*} 

If the incident field is ${ \alpha}$-quasi-periodic with period $\Lambda$, i.e., for any $j\in\Z$,
\begin{equation*}
u^i(x_1+\Lambda j,x_2)=e^{\i{\alpha}{\Lambda}{ j}}u^i(x_1,{ x}_2),
\end{equation*}
then the total field $u$ is also ${ \alpha}$-quasi-periodic. Thus for this case, the problem could be reduced into one periodic cell $\Omega^\Lambda_H$ naturally, such that $u$ satisfies the following equations.
\begin{eqnarray}\label{eq:quasi-per1}
\Delta u+k^2 u&=&0,\quad\text{ in }\Omega^\Lambda_H,\\
u&=&0,\quad\text{ on }\Gamma^\Lambda,\\
\frac{\partial u}{\partial x_2}(x_1,H)-T_{\alpha}^+\big[u|_{\Gamma_H}\big]&=&f,\quad \text{ on }\Gamma^\Lambda_H,\label{eq:quasi-per2}
\end{eqnarray}
where $T^+_{\alpha}$ is the ${\alpha}$-quasi-periodic DtN map defined in the form of
\begin{equation*}
T^+_{\alpha}(\phi)=\i\sum_{j\in\Z}\beta_j(\alpha)\hat{\phi}(j)e^{\i(\Lambda^*j+\alpha)x_1},\quad \phi=\sum_{ j\in\Z}\hat{\phi}( j)e^{\i(\Lambda^*j+\alpha)x_1},
\end{equation*}
where 
\begin{equation*}
\beta_j(\alpha)=\begin{cases}
\sqrt{k^2-|{\Lambda}^*{j}+{\alpha}|^2},\quad |{\Lambda}^*{ j}+{\alpha}|\leq k;\\
\i\sqrt{|{\Lambda}^*{ j}+{\alpha}|^2-k^2},\quad|{\Lambda}^*{ j}+{\alpha}|> k.
\end{cases}
\end{equation*}
$T^+_{\alpha}$ is the special form of $T^+$ for ${\alpha}$-quasi-periodic, and it is a bounded operator from $H^{1/2}_\alpha(\Gamma^\Lambda_H)$ to $H^{-1/2}_\alpha(\Gamma^\Lambda_H)$. 

The weak formulation for the quasi-periodic scattering problem has the following from for any $v\in \widetilde{H}^1_\alpha(\Omega^\Lambda_H)$ 
\begin{equation}\label{eq:var_per}
\int_{\Omega^\Lambda_H}\left[\grad u\cdot\grad\overline{v}-k^2 u\overline{v}\right]-\int_{\Gamma^\Lambda_H}T^+_\alpha \left[u\big|_{\Gamma^\Lambda_H}\right]\overline{v}\d s=\int_{\Gamma^\Lambda_H}f\overline{v}\d s
\end{equation}
The $\alpha$-quasi-periodic scattering problem is uniquely solved, that has been proved in \cite{Kirsc1993}.
\begin{theorem}\label{th:uni_period}
The $\alpha$-quasi-periodic scattering problem \eqref{eq:var_per} are uniquely solved in $\widetilde{H}_\alpha^1(\Omega^\Lambda_H)$, further more, there is a constant $C$ such that
\begin{equation*}
\|u\|_{\widetilde{H}^1_\alpha(\Omega^\Lambda_H)}\leq C\|u^i\|_{H^1_\alpha(\Omega^\Lambda_H)}.
\end{equation*}
\end{theorem}

The quasi-periodic scattering problems have been well-studied in the past 30 years, both theoretically and numerically. For numerical solutions of the problems we refer to \cite{George2011} for finite element methods, and \cite{Meier2000} for boundary integral equation methods. Thus it is possible to treat the numerical approximation of $u$ as a known function, and study the original problem by subtracting the approximated $u$ from the total field from locally perturbed surfaces, which will be discussed in the following section.

\section{Differences between two problems}

Let $\zeta_p$ be a local perturbation of $\zeta$, where the norm $\|\zeta_p-\zeta\|_{W^{1,\infty}(\R)}$ is bounded. For simplicity, we also assume that $\zeta_p(x_1)>0$ for any $x_1\in\R$. Define the sound soft surface $\Gamma_p$ by
\begin{equation*}
\Gamma_p:=\left\{(x_1,\zeta_p(x_1)):\,x_1\in\R\right\}.
\end{equation*}
\begin{remark}
For simplicity, we require that the local perturbation only exists in one periodic cell, i.e., $\rm{supp}(\zeta_p-\zeta)\subset \W\left(=\left(-\frac{\Lambda}{2},\frac{\Lambda}{2}\right]\right)$.
\end{remark}
To guarantee the well-posedness of the scattering problems, $\zeta_p$ is also assumed to be Lipschitz continuous. Similarly, we can also define the domain above $\Gamma_p$ by $\Omega^p$, and the restriction of $\Omega^p$ in the strip $\R\times[0,H]$ by $\Omega^p_H$. Here $H$ is assumed to be a constant number that is larger than both $\|\zeta\|_\infty$ and $\|\zeta_p\|_\infty$. If the incident field $u^i$ exists and satisfies the Helmholtz equation in $\Omega^p$, it is scattered by $\Gamma_p$. Denote the scattered field by $u^s_p$, the total field by $u_p$. Then the total field satisfies the equation
\begin{equation}
\Delta u_p+k^2 u_p=0\,\text{ in }\Omega_p,\quad u_p=0\,\text{ on }\Gamma_p,
\end{equation}
and the boundary condition \eqref{eq:boundary_condition} on $\Gamma_H$. $u_p$ satisfies the following variational problems, 
\begin{equation}\label{eq:var_ori}
\int_{\Omega^p_H}\left[\nabla u_p\cdot\nabla\overline{v}-k^2u_p\overline{v}\right]\d x-\int_{\Gamma_H}T^+\left[u_p\big|_{\Gamma_H}\right]\overline{v}\d s=\int_{\Gamma_H}f\overline{v}\d s
\end{equation}
for any $v\in \widetilde{H}_{loc}^1(\Omega_H)$ with compact support.

The variational form \eqref{eq:var_ori} is defined in the domain $\Omega^p_H$ while the quasi-periodic scattering problem is defined in the periodic domain $\Omega_H$, we have to reformulate the problems in the same domain. Define the diffeomorphism $\Phi_p:\,\Omega_{H_0}\rightarrow\Omega^p_{H_0}$ where $H_0$ is a positive number such that $\max\left\{\|\zeta\|_\infty,\|\zeta_p\|_\infty\right\}<H_0<H$, and it equals to the identity above $\Gamma_{H_0}$. An example of $\Phi_p$, which is also applied in the numerical examples in this paper, has the following representation
\begin{equation*}
\Phi_p:\, x\mapsto\left(x_1,\,x_2+\frac{(x_2-H)^3}{(\zeta(x_1)-H)^3}\left[\zeta_p(x_1)-\zeta(x_1)\right]\right).
\end{equation*}
 Let $u_T:=u_p\circ\Phi_p$, then $u_T$ is a function defined in the periodic domain  $\Omega_H$ and satisfies the variational equation for any $v\in\widetilde{H}^1_{loc}(\Omega_H)$ with a compact support
\begin{equation}\label{eq:var_trans}
\int_{\Omega_H}\left[A_p\grad u_T\cdot\grad\overline{v}-k^2 c_p u_T\overline{v}\right]\d {\bm x}-\int_{\Gamma_H}T^+\left[u_T\big|_{\Gamma_H}\right]\overline{v}\d s=\int_{\Gamma_H}f\overline{v}\d s,
\end{equation}
where $A_p$ and $c_p$ are defined by $\Phi_p$, i.e., 
\begin{eqnarray*}
&& A_p({\bm x})=\left|\det \grad\Phi_p({\bm x})\right|\left[\left(\Phi_p({\bm x})\right)^{-1}\left(\Phi_p({\bm x})\right)^{-\top}\right]\in \left( L^\infty(\Omega_H)\right)^{2\times 2};\\
&& c_p({\bm x})=\left|\det \grad\Phi_p({\bm x})\right|\in L^\infty(\Omega_H).
\end{eqnarray*}
As is assumed that the support of the perturbation $\zeta_p-\zeta$ exists in one periodic cell $\W$, the supports for $A_p-I_2$ (where $I_2$ is a $2\times 2$ identity matrix) and $c_p-1$ are both bounded domains, and they are both subsets of $\Omega^\Lambda_{H_0}$, i.e., $\rm{supp}(A_p-I_2),\,\rm{supp}(c_p-1)\subset\Omega^\Lambda_{H_0}\subset\Omega^\Lambda_H$.

From Theorem \ref{th:uni_solv}, the problem is uniquely solvable in $\widetilde{H}_r^1(\Omega_H^p)$ where $r$ could be any real number that $-1<r<-1/2$, thus $u_p\in \widetilde{H}_r^1(\Omega_H)$. As $u_T=u_p$ outside the domain $\Omega^\Lambda_H$, the field $u_T\in \widetilde{H}_r^1(\Omega_H)$, and also belongs to the space $H^1_{loc}(\Omega_H)$. Let $u_D:=u_T-u$, subtracting \eqref{eq:var_origional} from \eqref{eq:var_ori}, $u_D$ satisfies the following variational form for any $v\in \widetilde{H}^1(\Omega_H)$ with compact support
\begin{equation}\label{eq:var_diff}
\int_{\Omega_H}\left[A_p\grad u_D\cdot\grad\overline{v}-k^2 c_p u_D\overline{v}\right]\d { x}-\int_{\Gamma_H}T^+\left[u_D\big|_{\Gamma_H}\right]\overline{v}\d s=F(u,v)
\end{equation}
where the right hand side is defined by
\begin{equation*}
F(u,v)=\int_{\Omega_H}\left[(I_2-A_p)\grad u\cdot\grad\overline{v}-k^2 (1-c_p) u\overline{v}\right]\d { x}.
\end{equation*}
As the supports for both $A_p-I_2$ and $c_p-1$ are subsets of $\Omega^\Lambda_H$, $F(u,v)$ has the following representation
\begin{equation}\label{eq:rhs}
F(u,v)=\int_{\Omega_H^\Lambda}\left[(I_2-A_p)\grad u\cdot\grad\overline{v}-k^2 (1-c_p) u\overline{v}\right]\d { x}.
\end{equation}

The unique solvability of the variational problem \eqref{eq:var_diff} is described in the following theorem.

\begin{theorem}\label{th:uni_diff}
For any function $u\in H^1_{loc}(\Omega_H)$, the variational problem \eqref{eq:var_diff} is uniquely solvable in $H_r^1(\Omega_H)$, and there is a constant $C$ that does not depend on $u$ such that
\begin{equation*}
\|u_D\|_{\widetilde{H}_r^1(\Omega_H)}\leq C\|u^i\|_{H^1(\Omega^\Lambda_H)}.
\end{equation*}
\end{theorem}

\begin{proof}

From the representation of $F$ in \eqref{eq:rhs}, $F(u,v)$ is a bounded sesquilinear form, i.e., there is a constant $C$ that depends on $\|A_p-I_2\|_\infty$ and $\|c_p-1\|_\infty$ such that for any $r\in\R$,
\begin{equation*}
\left|F(u,v)\right|\leq C\|u\|_{H^1_r(\Omega^\Lambda_H)}\|v\|_{H^1_{-r}(\Omega_H)}\leq C\|u\|_{H^1(\Omega^\Lambda_H)}\|v\|_{H^1_{-r}(\Omega_H)}.
\end{equation*}

Following the proof in \cite{Chand2010}, the variational problem \eqref{eq:var_ori} is uniquely solvable in $\widetilde{H}_r^1(\Omega_H)$, if the right hand side is the conjugate of a bounded linear operator for $v\in H^1_{-r}(\Omega_H)$, for any $r\in(-1,1)$. From the equivalence of the left hand side in \eqref{eq:var_diff} and \eqref{eq:var_ori}, the variational form \eqref{eq:var_diff} 
 is uniquely solvable in $\widetilde{H}_r^1(\Omega_H)$ for any $|r|<1$, and the solution $u_D$ is bounded by the incident field $u^i$, i.e., there is a constant $C>0$ that does not depend on the incident field such that $\|u_D\|_{\widetilde{H}_r^1(\Omega_H)}\leq C\|u^i\|_{H_\alpha^1(\Omega^\Lambda_H)}$. The proof is finished.
\end{proof}

From the arguments in this section, we can see that the difference of two total fields $u_D$ belongs to the space $\widetilde{H}_r^1(\Omega_H)$, for any $|r|<1$. Compared to the total field $u_p\in\widetilde{H}_{r_1}^1(\Omega_H)$ for $-1<r_1<-1/2$, $u_D$ decays much faster, thus it is possible to be analysed by the Bloch transform, and be solved by the finite element method.

\section{The Bloch Transform of $u_D$}

In this section, we will apply the Bloch transform to the variational form \eqref{eq:var_diff}. Firstly, let's write the equation into the equivalent formulation, i.e.,
\begin{equation}
\int_{\Omega_H}\left[\grad u_D\cdot\grad\overline{v}-k^2 u_D\overline{v}\right]\d x-F(u_D,v)-\int_{\Gamma_H} T^+\big[u_D|_{\Gamma_H}\big]\overline{v}\d s=F(u,v)
\end{equation}
Let $w(\alpha,\cdot)=\left(\J_\Omega u_D\right)(\alpha,\cdot)$, using the property that the Bloch transform is an isomorphism between $\widetilde{H}_r^1(\Omega_H)$ and $H_0^r(\Wast;\widetilde{H}^1_\alpha(\Omega^\Lambda_H))$ (see Appendix) with $\J_\Omega^{-1}=\J_\Omega^*$, and it commutes with the partial derivatives, the first term could be written into
\begin{equation*}
\begin{aligned}
&\int_{\Omega_H}\left[\grad u_D\cdot\grad\overline{v}-k^2 u_D\overline{v}\right]\d x\\
=&\int_{\Omega_H}\left[(\J_\Omega^{-1}\circ\J_\Omega)\grad u_D\cdot\grad\overline{v}-k^2 (\J_\Omega^{-1}\circ\J_\Omega)u_D\overline{v}\right]\d x\\
=&\int_{\Omega_H}\left[(\J_\Omega^*\circ\J_\Omega)\grad u_D\cdot\grad\overline{v}-k^2 (\J_\Omega^*\circ\J_\Omega)u_D\overline{v}\right]\d x\\
=&\int_\Wast\int_{\Omega^\Lambda_H}\left[\grad \left[\J_\Omega u_D\right]\cdot\J_\Omega\left[\grad\overline{v}\right]-k^2\left[ \J_\Omega u_D\right]\left[\J_\Omega\overline{v}\right]\right]\d\alpha\d x\\
=&\int_\Wast\int_{\Omega^\Lambda_H}\left[\grad w\cdot\grad\left[\J_\Omega\overline{v}\right]-k^2w\left[\J_\Omega\overline{v}\right]\right]\d\alpha\d x
\end{aligned}
\end{equation*}
Let $\phi=\overline{\J_\Omega\overline{v}}$, we can finally arrive at
\begin{equation*}
\int_{\Omega_H}\left[\grad u_D\cdot\grad\overline{v}-k^2 u_D\overline{v}\right]\d x=\int_\Wast\int_{\Omega^\Lambda_H}\left[\grad w\cdot\grad\overline{\phi}-k^2w\overline{\phi}\right]\d\alpha\d x.
\end{equation*}
Then $w$ satisfies the following variational form for $\phi(\alpha,\cdot)=\left(\J_\Omega v\right)(\alpha,\cdot)$
\begin{equation}\label{eq:var_Bloch}
\int_\Wast a_\alpha(w(\alpha,\cdot),\phi(\alpha,\cdot))\d \alpha-F(\J_\Omega^{-1} w,\J_\Omega^{-1}\phi)=F(u,\J_\Omega^{-1}\phi), 
\end{equation}
where $a_\alpha(\cdot,\cdot)$ has the representation defined for two $\alpha$-quasi-periodic functions $\xi_1,\xi_2$
\begin{equation*}
a_\alpha(\xi_1,\xi_2)=\int_{\Omega^\Lambda_H}\left[\grad \xi_1\cdot\grad\overline{\xi_2}-k^2 \xi_1\overline{\xi_2}\right]\d x-\int_{\Gamma^\Lambda_H}T^+_\alpha\left[\xi_1\big|_{\Gamma^\Lambda_H}\right]\overline{\xi_2}\d s.
\end{equation*}

From the procedure we just went through, it is easy to prove the equivalence between \eqref{eq:var_diff} and \eqref{eq:var_Bloch}. 

\begin{theorem}\label{th:equivalence}
Given a quasi-periodic incident field $u^i$, let $u_D:=u_p\circ\Phi_p-u$, then $u_D\in \widetilde{H}_r^1(\Omega_H)$ for any $|r|<1$ satisfies \eqref{eq:var_diff} if and only if $w:=\J_\Omega u_D\in H_0^r(\Wast;\widetilde{H}_\alpha^1(\Omega^\Lambda_H))$ satisfies \eqref{eq:var_Bloch}. 
\end{theorem}

With the unique solvability of the variational problem \eqref{eq:var_ori} in Theorem \ref{th:uni_diff} and the equivalence shown in Theorem \eqref{th:equivalence}, we can prove the unique solvability of the variational problem \eqref{eq:var_Bloch}.

\begin{theorem}\label{th:unique}
When $\Gamma,\,\Gamma_p$ are graphs of Lipschitz continuous functions $\zeta,\,\zeta_p$, the variational form \eqref{eq:var_Bloch} is uniquely solvable in $H_0^r(\Wast;\widetilde{H}_\alpha^1(\Omega^\Lambda_H))$ for any $|r|<1$, for any quasi-periodic incident field $u^i\in H_\alpha^1(\Omega^\Lambda_H)$. 
\end{theorem}

\begin{proof}
From Theorem \ref{th:equivalence}, the variational problem \eqref{eq:var_Bloch} is equivalent to \eqref{eq:var_diff}. From Theorem \ref{th:uni_diff}, the problem \eqref{eq:var_diff} is uniquely solvable in $H_r^1(\Omega_H)$ for any $|r|<1$, thus the problem \eqref{eq:var_Bloch} is uniquely solvable.
\end{proof}

Following the proofs of Theorem 7 and Theorem 8 in \cite{Lechl2017}, when the surfaces are smoother, the solution has a higher regularity.

\begin{theorem}\label{th:smooth}
Suppose $\zeta,\zeta_p\in C^{2,1}(\R)$, then the unique solution of \eqref{eq:var_Bloch} $w$ belongs to the space $H_0^r(\Wast;\widetilde{H}_\alpha^2(\Omega^\Lambda_H))$ for any $1/2<r<1$ and $u_D=\J^{-1}_\Omega w$ belongs to $\widetilde{H}^2_r(\Omega_H)$. Moreover, the problem \eqref{eq:var_Bloch} equivalently satisfies the following family of problems for all $\alpha\in\Wast$ and $v_\alpha\in \widetilde{H}_\alpha^1(\Omega^\Lambda_H)$
\begin{equation}
a_\alpha(w(\alpha,\cdot),v_\alpha)-F(\J^{-1}_\Omega w,v_\alpha)=F(u,v_\alpha).
\end{equation} 
\end{theorem}

\section{Standard Finite Element Method}

In this section, we will introduce the standard finite element method for the scattering problems, and then discuss the error estimation for the numerical scheme. To apply the Galerkin discretization method, we need to define the meshes and finite element spaces at first. 

For the bounded periodic cell $\Omega^\Lambda_H$, let $\mathcal{M}_h$ be a family of regular and quasi-uniform meshes for it, where $0<h\leq h_0$ is the maximum mesh width, where $h_0>0$ is a small enough positive number. For convenience, let the nodal points on the left and right boundaries of $\Omega^\Lambda_H$ have the same heights. Thus we can set up the piecewise linear and globally continuous functions on this mesh. By omitting all the nodal points on the left boundary of $\Omega^\Lambda_H$, we can construct the piecewise linear and globally continuous functions on the mesh $\mathcal{M}_h$ that could be extended periodically to be continuous  functions in the periodic domain $\Omega_H$. Define the function that equals to the $\ell$-th nodal point and zero at other nodal points by $\left\{\phi^{(\ell)}_M\right\}_{\ell=1}^M$, and define the discrete subspace by $V_h^0:=\rm{span}\left\{\phi^{(1)}_M,\dots,\phi^{(M)}_M\right\}$, thus $V_h^0$ is a subspace of $\widetilde{H}_0^1(\Omega^\Lambda_H)$. We can also define the quasi-periodic function space by
\begin{equation*}
V_{h}^{\alpha}:=\left\{e^{\i\alpha x_1}v_h:\,v_h\in V^0_h\right\}\subset\widetilde{H}^1_\alpha(\Omega^\Lambda_H).
\end{equation*}
To introduce the finite element space in $\Wast$, we define the uniformly distributed grid points of the interval $\Wast=\left(-\frac{\pi}{\Lambda},\frac{\pi}{\Lambda}\right]$
\begin{equation*}
\alpha^{(1)}_N=-\frac{\pi}{\Lambda}+\frac{\pi}{N\Lambda},\,\alpha^{(j)}_N=\alpha^{(j-1)}_N+\frac{2\pi}{N\Lambda},\quad j=2,\dots,N,
\end{equation*}
then define the piecewise constant basic function as $\left\{\psi^{(j)}_N\right\}_{j=1}^N$ that equals to one  on the interval $[\alpha_N^{(j)}-\pi/(N\Lambda),\alpha_N^{(j)}+\pi/(N\Lambda)]$ and zero otherwise. Thus we can define the finite element space in the 3D domain $\Wast\times\Omega^\Lambda_H$ by
\begin{equation*}
X_{N,h}=\left\{v_{N,h}(\alpha,x)=e^{-\i\alpha x_1}\sum_{j=1}^N\sum_{\ell =1}^M v_{N,h}^{(j,\ell)}\psi_N^{(j)}(\alpha)\psi_M^{(\ell)}(x):\,v_{N,h}^{(j,\ell)}\in\C\right\}\subset L^2(\Wast;\widetilde{H}^1_\alpha(\Omega^\Lambda_H)).
\end{equation*}

Now with the definitions of the finite element spaces, we can introduce the numerical scheme to solve the scattering problem \eqref{eq:sca1}, \eqref{eq:sca2} and \eqref{eq:boundary_condition} with the quasi-periodic incident field $u^i$. Based on the procedure in the above sections, it could be briefly divided into three steps.
\begin{algorithm}\label{alg}Standard finite element method for the scattering problems.
\begin{enumerate}
\item Find $u_h\in V_h^\alpha$ that solves the following variational problem for all $v_h\in V^\alpha_h$
\begin{equation}\label{eq:fem_var_per}
a_\alpha(u_h,v_h)=\int_{\Gamma^\Lambda_H}f\overline{v_h}\d s.
\end{equation}
\item Find $w_{N,h}\in X_{N,h}$ that solves the following variational problem for all $\phi_{N,h}\in X_{N,h}$
\begin{equation}\label{eq:fem_var_Bloch}
\int_\Wast a_\alpha(w_{N,h},\phi_{N,h})\d \alpha+F(\J_\Omega^{-1} w_{N,h},\J_\Omega^{-1} \phi_{N,h})=F(u_h,\J_\Omega^{-1}\phi_{N,h}).
\end{equation}

\item Let $u_{D}^{N,h}=\J_\Omega^{-1} w_{N,h}$ and then the approximation of $u_T$ is obtained by $u_{T}^{N,h}=u_h+u_{D}^{N,h}$.
\end{enumerate}
\end{algorithm}
As functions in the space ${X}_{N,h}$ are piecewise exponential on $\alpha$ for any fixed $x_1$, we can explicitly get the form of the inverse Bloch transforms of any $w_{N,h}\in X_{N,h}$ 
\begin{equation*}
\begin{aligned}
\J_\Omega^{-1}w_{N,h}(\alpha,x)&=\left[\frac{\Lambda}{2\pi}\right]^{1/2}\sum_{j=1}^N\int_{\alpha^{(j)}_N-\pi/(N\Lambda)}^{\alpha^{(j)}_N+\pi/(N\Lambda)}\d\alpha\\
&=\left[\frac{\Lambda}{2\pi}\right]^{1/2}\sum_{j=1}^N\sum_{\ell=1}^M v_{N,h}^{(j,\ell)}\phi_M^{(\ell)}(x)\int_{\alpha^{(j)}_N-\pi/(N\Lambda)}^{\alpha^{(j)}_N+\pi/(N\Lambda)}e^{-\i\alpha x_1}\d\alpha\\
&=\left[\frac{\Lambda}{2\pi}\right]^{1/2}\sum_{j=1}^N g_N^{(j)}(x_1)\sum_{\ell=1}^M  v_{N,h}^{(j,\ell)}\phi_M^{(\ell)}(x),
\end{aligned}
\end{equation*}
where 
\begin{equation*}
g_N^{(j)}(x_1)=\int_{\alpha^{(j)}_N-\pi/(N\Lambda)}^{\alpha^{(j)}_N+\pi/(N\Lambda)}e^{-\i\alpha x_1}\d\alpha=\begin{cases}2 e^{-\i\alpha_N^{(j)}x_1}\sin(\pi x_1/(N\Lambda))/x_1,\quad\text{ when }x_1\neq 0;\\
2\pi/(N\Lambda),\quad\text{ when }x_1=0.
\end{cases}
\end{equation*}
Thus it defines the discrete form of the inverse Bloch transform that will be applied for the numerical schemes in the following sections, i.e.,
\begin{equation}
\J_{\Omega,N}^{-1}\left(\left\{w_{N,h}(\alpha^{(j)}_N,\cdot)\right\}_{j=1}^N\right):=\left[\frac{\Lambda}{2\pi}\right]^{1/2}\sum_{j=1}^N g_N^{(j)}(x_1)\sum_{\ell=1}^M  v_{N,h}^{(j,\ell)}\phi_M^{(\ell)}(x).
\end{equation}

Now we are prepared to work on the error estimation of Algorithm \ref{alg}. For the first step, that solves a quasi-periodic scattering problem, could be investigated based on the arguments from \cite{George2011,Lechl2016a}. From Theorem 14 in \cite{Lechl2016a}, we can get the error estimation for the finite element approximation of  \eqref{eq:quasi-per1}-\eqref{eq:quasi-per2} in ${V}^\alpha_h$.

\begin{theorem}\label{th:fem_quasi_per}
Suppose $\Omega_H$ is a domain of class $C^{1,1}$, if $u^i\in H_\alpha^{3/2}(\Gamma^\Lambda_H)$, then the solution $u\in\widetilde{H}^2_\alpha(\Omega^\Lambda_H)$, and the numerical result in the space, denoted by $u_h$, that solves
\begin{equation}
a_\alpha(u_h,v_h)=\int_{\Gamma^\Lambda_H}f\overline{v_h}\d s\quad\text{ for all }v_h\in\widetilde{V}^\alpha_h,
\end{equation}
satisfies for some positive constant $C$ that does not depend on $u^i$
\begin{equation*}
\|u_h-u\|_{H_\alpha^{\ell}(\Omega^\Lambda_H)}\leq Ch^{2-\ell}\|u^i\|_{H^{3/2}_\alpha(\Gamma^\Lambda_H)},\quad\ell=0,1.
\end{equation*}
\end{theorem}

From the error estimation of $u_h$ and $u$, when $h$ is small enough, we can simply pick up the constant $C$ that does not depend on either $u^i$ or $h$ such that $\|u_h\|_{H^1_\alpha(\Omega^\Lambda_H)}\leq C\|u^i\|_{H^{3/2}_\alpha(\Gamma^\Lambda_H)}$. Then we can turn to Step 2 in Algorithm \ref{alg}. Let's define the sesquilinear form on $L^2(\Wast;\widetilde{H}^1_\alpha(\Omega^\Lambda_H))$ by
\begin{equation*}
\mathcal{A}(w,\phi):=\int_\Wast a_\alpha\left(w(\alpha,\cdot),\phi(\alpha,\cdot)\right)\d\alpha-F(\J_\Omega^{-1}w,\J_\Omega^{-1}\phi),
\end{equation*}
thus $w$ and $w_{N,h}$ are the solutions of the variational forms
\begin{eqnarray}\label{eq:var1}
&&\mathcal{A}(w,\phi)=F(u,\phi),\quad w,\,\phi\in L^2(\Wast;\widetilde{H}^1_\alpha(\Omega^\Lambda_H)),\\
&&\mathcal{A}(w_{N,h},\phi_{N,h})=F(u_h,\phi_{N,h}),\quad w_{N.h},\,\phi_{N,h}\in X _{N,h}.\label{eq:var2}
\end{eqnarray}
Let $w_h$ be the exact solution of the variational problem
\begin{eqnarray}\label{eq:var3}
&&\mathcal{A}(z,\phi)=F(u_h,\phi),\quad z,\,\phi\in L^2(\Wast;\widetilde{H}^1_\alpha(\Omega^\Lambda_H)),
\end{eqnarray}
then the next work is to study the convergence rate of $w_{N,h}$ to $z$ with respect to $N$, when $w_h$ is any piecewise linear and globally continuous function in $\Omega^\Lambda_H$. To this end, we have to introduce the variational problem
\begin{equation}\label{eq:var4}
\mathcal{A}(z_{N,h},\phi_{N,h})=F(u,\phi_{N,h}),\quad w,\,\phi_{N,h}\in X _{N,h}.
\end{equation}

From the result that $u\in\widetilde{H}^2_\alpha(\Omega^\Lambda_H)$ and the proof of Theorem 9 in \cite{Lechl2017}, we have the following  convergence property of $z_{N,h}$ to $w$.

\begin{theorem}\label{th:err1}
When $\zeta,\zeta_p\in C^{2,1}(\R)$, then the variational problem \eqref{eq:var4} is uniquely solvable in $ X _{N,h}$ for any $u\in \widetilde{H}^2_\alpha(\Omega^\Lambda_H)$. When $N\in\N$ is large enough and $0<h\leq h_0$ is small enough, the numerical solution satisfies the error estimate
\begin{equation}
\|z_{N,h}-w\|_{L^2(\Wast;\widetilde{H}^\ell(\Omega^\Lambda_H))}\leq Ch^{1-\ell}(N^{-r}+h)\|u\|_{\widetilde{H}^2(\Omega^\Lambda_H)},\quad\ell=0,1.
\end{equation}
\end{theorem}

Let $\eta:=w-z$ and $\eta_{N,h}:=z_{N,h}-w_{N,h}$, then they are solutions of
\begin{eqnarray}\label{eq:var5}
&& \mathcal{A}(\eta,\phi)=F(u-u_h,\phi),\quad \eta,\,\phi\in L^2(\Wast;\widetilde{H}^1_\alpha(\Omega^\Lambda_H));\\
&& \mathcal{A}(\eta_{N,h},\phi)=F(u-u_h,\phi_{N,h}),\quad \eta_{N,h},\,\phi_{N,h}\in X _{N,h}.
\end{eqnarray}
From the well-posedness of the problem \eqref{eq:var5} and Theorem \ref{th:fem_quasi_per}, there is a constant $C$ such that
\begin{equation}\label{eq:err_per}
\|\eta\|_{L^2(\Wast;\widetilde{H}^1_\alpha(\Omega^\Lambda_H))}\leq C\|u-u_h\|_{H^1(\Omega^\Lambda_H)}\leq Ch\|u^i\|_{H^{3/2}_\alpha(\Gamma^\Lambda_H)}.
\end{equation}
However, the converegence of $\eta_{N,h}$ could not be easily obtained from the similar argument of Theorem \ref{th:err1}, for $u-u_h$ is only an $H^1$-function. However, we can also estimate the error by the  Aubin-Nitsche-Lemma (see \cite{Brenn1994}).

\begin{theorem}\label{th:err2}
When $N\in\N$ is large enough and $h>0$ is small enough, there is a constant $C>0$ such that 
\begin{equation}
\|\eta-\eta_{N,h}\|_{L^2(\Wast;\widetilde{H}^\ell_\alpha(\Omega^\Lambda_H))}\leq Ch^{2-\ell}\|u^i\|_{H^{3/2}_\alpha(\Gamma^\Lambda_H)}.
\end{equation}
\end{theorem}

\begin{proof} 1) Firstly, let's consider the $L^2$-estimation. 
Consider the dual problem, i.e., for any $\phi\in L^2(\Wast\times\Omega^\Lambda_H)$, to find the solution $w^\phi\in L^2(\Wast;\widetilde{H}^1_\alpha(\Omega^\Lambda_H))$ of the variational problem
\begin{equation*}
\mathcal{A}(\psi,w^\phi)=\int_\Wast\int_{\Omega^\Lambda_H}\psi\overline{\phi}\d x\d s \text{ for all }\psi\in L^2(\Wast;\widetilde{H}_\alpha^1(\Omega^\Lambda_H)).
\end{equation*}
From the well-posedness of the variational problem \eqref{eq:var1}, the well-posedness of the dual problem could be obtained from the conjugation of the problem. From the interior regularity for elliptic equations, $w^\phi\in L^2(\Wast;\widetilde{H}^2_\alpha(\Omega^\Lambda_H))$ and bounded:
\begin{equation*}
\|w^\phi\|_{L^2(\Wast;\widetilde{H}^2_\alpha(\Omega^\Lambda_H))}\leq C\|\phi\|_{L^2(\Omega^\Lambda_H)}.
\end{equation*}
Thus by replacing $\psi$ by $\eta-\eta_{N,h}$,
\begin{equation*}
\int_\Wast\int_{\Omega^\Lambda_H}{(\eta-\eta_{N,h})}\overline{\phi}\d x\d s =\mathcal{A}(\eta-\eta_{N,h},w^\phi).
\end{equation*}
From the Galerkin orthogonality, for any piecewise linear function $w^\phi_{N,h}$,
\begin{equation*}
\mathcal{A}(\eta-\eta_{N,h},w^\phi_{N,h})=0,
\end{equation*}
then
\begin{equation*}
\int_\Wast\int_{\Omega^\Lambda_H}{(\eta-\eta_{N,h})}\overline{\phi}\d x\d s =\mathcal{A}(\eta-\eta_{N,h},w^\phi-w^\phi_{N,h}).
\end{equation*}
Thus 
\begin{equation*}
\begin{aligned}
&\|\eta-\eta_{N,h}\|_{L^2(\Wast\times\Omega^\Lambda_H)}=\sup_{\|\phi\|_{L^2(\Wast\times\Omega^\Lambda_H)}=1}\left[\int_\Wast\int_{\Omega^\Lambda_H}(\eta-\eta_{N,h})\overline{\phi}\d x\d s\right]\\
\leq &\|\eta-\eta_{N,h}\|_{L^2(\Wast;\widetilde{H}^1_\alpha(\Omega^\Lambda_H))}\sup_{\|\phi\|_{L^2(\Wast\times\Omega^\Lambda_H)}=1}\inf_{w^\phi_{N,h}\in X _{N,h}}\|w^\phi-w^\phi_{N,h}\|_{L^2(\Wast;\widetilde{H}^1_\alpha(\Omega^\Lambda_H))}.
\end{aligned}
\end{equation*}
As $\phi\in L^2(\Wast\times\Omega^\Lambda_H)$, we have the following error estimate
\begin{equation*}
\inf_{w_{N,h}^\phi\in X _{N,h}}\|w^\phi_{N,h}-w^\phi\|_{L^2(\Wast;\widetilde{H}^1_\alpha(\Omega^\Lambda_H))}\leq Ch\|w^\phi\|_{L^2(\Wast;\widetilde{H}^2_\alpha(\Omega^\Lambda_H))}\leq Ch\|\phi\|_{L^2(\Wast\times\Omega^\Lambda_H)}.
\end{equation*}
Thus from the projection that $\|\eta-\eta_{N,h}\|_{L^2(\Wast;\widetilde{H}^1_\alpha(\Omega^\Lambda_H))}\leq C\|\eta\|_{L^2(\Wast;\widetilde{H}^1_\alpha(\Omega^\Lambda_H))}$, and the estimation that \eqref{eq:err_per},
\begin{equation*}
\|\eta-\eta_{N,h}\|_{L^2(\Wast\times\Omega^\Lambda_H)}\leq Ch\|\eta\|_{L^2(\Wast;\widetilde{H}^1_\alpha(\Omega^\Lambda_H))}\leq Ch^2\|u^i\|_{H^{3/2}_\alpha(\Gamma^\Lambda_H)}.
\end{equation*}

2) Then let's consider the $H^1$-estimation. The only difference is to take $\phi\in L^2(\Wast;H^{-1}(\Omega^\Lambda_H))$. From the well-posedness and regularity, $w^\phi\in L^2(\Wast;\widetilde{H}^1_\alpha(\Omega^\Lambda_H))$ and satisfies
\begin{equation*}
\|w^\phi\|_{L^2(\Wast;\widetilde{H}^1_\alpha(\Omega^\Lambda_H))}\leq C\|\phi\|_{L^2(\Wast;H^{-1}(\Omega^\Lambda_H))}.
\end{equation*}
Thus
\begin{equation*}
\begin{aligned}
&\|\eta-\eta_{N,h}\|_{L^2(\Wast;\widetilde{H}^1_\alpha(\Omega^\Lambda_H))}=\sup_{\|\psi\|_{L^2(\Wast;H^{-1}(\Omega^\Lambda_H))}=1}\left[\int_\Wast\int_{\Omega^\Lambda_H}{(\eta-\eta_{N,h})}\overline{\phi}\d x\d s\right]\\
\leq &\|\eta-\eta_{N,h}\|_{L^2(\Wast;\widetilde{H}^1_\alpha(\Omega^\Lambda_H))}\sup_{\|\psi\|_{L^2(\Wast;H^{-1}(\Omega^\Lambda_H))}=1}\inf_{w^\phi_{N,h}\in X _{N,h}}\|w^\phi-w^\phi_{N,h}\|_{L^2(\Wast;\widetilde{H}^1_\alpha(\Omega^\Lambda_H))}.
\end{aligned}
\end{equation*}
From similar argument,
\begin{equation*}
\inf_{w_{N,h}^\phi\in X _{N,h}}\|w^\phi_{N,h}-w^\phi\|_{L^2(\Wast;\widetilde{H}^1_\alpha(\Omega^\Lambda_H))}\leq C\|w^\phi\|_{L^2(\Wast;\widetilde{H}^1_\alpha(\Omega^\Lambda_H))}\leq C\|\phi\|_{L^2(\Wast;H^{-1}(\Omega^\Lambda_H))}.
\end{equation*}
Thus in the same way,
\begin{equation*}
\|\eta-\eta_{N,h}\|_{L^2(\Wast;\widetilde{H}^1_\alpha(\Omega^\Lambda_H))}\leq C\|\eta\|_{L^2(\Wast;\widetilde{H}^1_\alpha(\Omega^\Lambda_H))}\leq Ch\|u^i\|_{H^{3/2}_\alpha(\Gamma^\Lambda_H)}.
\end{equation*}
The proof is finished.
\end{proof}

With the results in Theorem \ref{th:err1} and Theorem \ref{th:err2}, we can get the convergence of Step 2 in Algorithm \ref{alg}.
\begin{theorem}\label{th:err3}
The error between $w_{N,h}$ and $z$ is bounded by
\begin{equation*}
\|w_{N,h}-z\|_{L^2(\Wast;\widetilde{H}^\ell(\Omega^\Lambda_H))}\leq Ch^{1-\ell}(N^{-r}+h)\|u^i\|_{H^{3/2}_\alpha(\Omega^\Lambda_H)}.
\end{equation*}
\end{theorem}

\begin{proof}
The result simply comes from the triangle inequality, i.e.,
\begin{equation*}
\begin{aligned}
\|w_{N,h}-z\|_{L^2(\Wast;\widetilde{H}^\ell(\Omega^\Lambda_H))}&=\|w_{N,h}-z_{N,h}-z+w+z_{N,h}-w\|_{L^2(\Wast;\widetilde{H}^\ell(\Omega^\Lambda_H))}\\
&\leq\|\eta-\eta_{N,h}\|_{L^2(\Wast;\widetilde{H}^\ell(\Omega^\Lambda_H))}+\|z_{N,h}-w\|_{L^2(\Wast;\widetilde{H}^\ell(\Omega^\Lambda_H))}\\
&\leq C h^{2-\ell}\|u^i\|_{H^{3/2}_\alpha(\Gamma^\Lambda_H)}+Ch^{1-\ell}(N^{-r}+h)\|u\|_{\widetilde{H}^2(\Omega^\Lambda_H)}\\&\leq Ch^{1-\ell}(N^{-r}+h)\|u^i\|_{H^{3/2}_\alpha(\Gamma^\Lambda_H)}.
\end{aligned}
\end{equation*}
The proof is finished.
\end{proof}

Theorem \ref{th:err3} has shown the convergence rate of Step 2. For the error estimation of Algorithm, we have to further study the difference $w_{N,h}-w$. 
From the regularity we have
\begin{equation*}
\|w-z\|_{L^2(\Wast;\widetilde{H}^1_\alpha(\Omega^\Lambda_H))}\leq C \|u-u_h\|_{H^1(\Omega^\Lambda_H)}\leq Ch\|u^i\|_{H^{3/2}_\alpha(\Gamma^\Lambda_H)}.
\end{equation*}
Thus it is easy to obtain the $H^1$-error estimation for $w_{N,h}$.

\begin{theorem}
The error between $w_{N,h}$ and $w$ is bounded by
\begin{equation*}
\|w_{N,h}-w\|_{L^2(\Wast;\widetilde{H}^1(\Omega^\Lambda_H))}\leq C (N^{-r}+h)\|u^i\|_{H^{3/2}_\alpha(\Omega^\Lambda_H)}.
\end{equation*}
\end{theorem}

\begin{remark}
From the results we have obtained, the algorithm converges at the rate of $h(N^{-r}+h)$ with respect to the $L^2(\Wast\times\Omega^\Lambda_H)$ norm, but we can only prove the error bound $O(N^{-r}+h)$ with respect to the $L^2(\Wast;\widetilde{H}^1(\Omega^\Lambda_H))$ norm. The reason for that is the $H^1$ function $u_h$. If we can replace $u_h$ by the $H^2$-approximation, i.e., the numerical result from the  global $C^1$-finite element method, we can easily obtain the error bound of $h(N^{-r}+h)$ with respect to the $L^2(\Wast\times\Omega^\Lambda_H)$ norm.
\end{remark}

\section{High order method}

In this section, we will introduce the high order numerical method to solve the plane waves scattering from locally perturbed periodic surfaces.  By going deeper into the regularity of $w$ with respect to $\alpha$, as was introduced in \cite{Zhang2017d}, we can adopt the high order method in \cite{Zhang2017e} for the problems.

\subsection{Further study of $w(\alpha,\cdot)$}

Firstly, we will periodize the quasi-periodic functions $w(\alpha,\cdot)$, so the functions will belong to the same periodic function space, and the quasi-periodicities will be implied in the sesquilinear forms.

Define the following two functions
\begin{equation*}
w_0(\alpha,x)=e^{\i\alpha x_1}w(\alpha,x),\quad v_\alpha(x)=e^{\i\alpha x_1}v_0(x),
\end{equation*}
where $v_0$ is a function that belongs to $\widetilde{H}_0^1(\Omega^\Lambda_H)$, and $w_0,\phi_0\in L^2(\Wast;\widetilde{H}^1_0(\Omega^\Lambda_H))$. Thus from \eqref{eq:var_Bloch} and Theorem \ref{th:smooth}, $w_0$ satisfies the following variational equation for any test function $v_0\in\widetilde{H}^1_0(\Omega^\Lambda_H)$
\begin{equation}\label{eq:var_periodize}
\begin{aligned}
\int_{\Omega^\Lambda_H}\left[(\grad+\i\alpha{\bm e}_1)w_0\cdot(\grad-\i\alpha{\bm e}_1)\overline{v_0}-k^2 w_0\overline{v_0}\right]\d x-\int_{\Gamma^\Lambda_H}\widetilde{T}^+_\alpha\left[w_0\big|_{\Gamma^\Lambda_H}\right]\overline{v_0}\d s\\
=\int_{\Omega^\Lambda_H}e^{-\i\alpha x_1}\left[(I_2-A_p)\grad u_T\cdot(\grad-\i\alpha{\bm e}_1)\overline{v_0}-k^2(1-c_p)u_T\overline{v_0}\right]\d x,
\end{aligned}
\end{equation}
where ${\bm e}_1=(1,0)^\top$, and $\widetilde{T}^+_\alpha$ is the modified Dirichlet-to-Neumann map defined on periodic functions, i.e,
\begin{equation*}
\widetilde{T}^+_\alpha(\phi)=\i\sum_{j\in\Z}\sqrt{k^2-|\Lambda^*j-\alpha|^2}\hat{\phi}(j)e^{\i j\Lambda x_1},\quad \phi=\sum_{j\in\Z}\hat{\phi}(j)e^{\i j\Lambda x_1}.
\end{equation*}
Let's define the sesquilinear forms by
\begin{eqnarray*}
&& a_\alpha(\xi_1,\xi_2):=\int_{\Omega^\Lambda_H}\left[(\grad+\i\alpha{\bm e}_1)\xi_1\cdot(\grad-\i\alpha{\bm e}_1)\overline{\xi_2}-k^2\xi_1\overline{\xi_2}\right]\d x,\\
&& b_\alpha(\xi_1,\xi_2):=\int_{\Gamma^\Lambda_H}\widetilde{T}^+_\alpha\left[\xi_1\big|_{\Gamma^\Lambda_H}\right]\overline{\xi_2}\d s,\\
&& F_\alpha(\xi_1,\xi_2):=\int_{\Omega^\Lambda_H}e^{-\i\alpha x_1}\left[(I_2-A_p)\grad \xi_1\cdot(\grad-\i\alpha{\bm e}_1)\overline{\xi_2}-k^2(1-c_p)\xi_1\overline{\xi_2}\right]\d x,
\end{eqnarray*}
then $w_0$ satisfies the variational equation for any $v_0\in\widetilde{H}^1_0(\Omega^\Lambda_H)$
\begin{equation}
a_\alpha(w_0(\alpha,\cdot),v_0)+b_\alpha(w_0(\alpha,\cdot),v_0)=F_\alpha(u_T,v_0).
\end{equation}

The terms both $a_\alpha(w_0(\alpha,\cdot),v_0)$ and $F_\alpha(u_T,v_0)$ depend analytically on $\alpha$. The term $b_\alpha(w_0(\alpha,\cdot),v_0)$ depends analytically on $\alpha$ except for some discrete points, and near each of these discrete points, there is a square-root like singularity. Thus following the arguments in \cite{Kirsc1993}, we can also obtain the dependence of $w_0(\alpha,\cdot)$ on $\alpha$. Firstly, we will introduce the set $\mathcal{S}_\Lambda$ by
\begin{equation*}
\mathcal{S}_\Lambda:=\left\{\alpha\in\overline{\Wast}:\,\exists\, n\in\Z,\text{ s.t., }|\alpha+n\Lambda^*|=k\right\}.
\end{equation*}
Then by the results in \cite{Kirsc1993,Zhang2017d}, the following theorem is easy to be proved.

\begin{theorem}
\label{th:regularity}
$w(\alpha,\cdot)$ depends analytically on $\alpha$ in $\Wast\setminus\mathcal{S}_\Lambda$, and for any $\alpha_0\in \mathcal{S}_\Lambda$, then there is a neighbourhood $U$ of $\alpha_0$ such that for any $\alpha\in U$, there are two functions $w_1(\alpha,\cdot)$ and $w_2(\alpha,\cdot)$ that depends analytically on $U\cap\Wast$ such that
\begin{equation*}
w(\alpha,\cdot)=w_1(\alpha,\cdot)+\sqrt{\alpha-\alpha_0}\,w_2(\alpha,\cdot).
\end{equation*}
\end{theorem}

Or the following corollary holds.

\begin{corollary}
For any $[\alpha_1,\alpha_2]$ such that $\alpha_1,\,\alpha_\in\mathcal{S}_\Lambda$ and $(\alpha_1,\alpha_2(\cap\mathcal{S}_\Lambda=\emptyset$, then there are three functions $v_1(\alpha,\cdot),\,v_2(\alpha,\cdot),\,v_3(\alpha,\cdot)$ that depend smoothly on $\alpha\in[\alpha_1,\alpha_2]$ such that
\begin{equation}
w(\alpha,\cdot)=v_1(\alpha,\cdot)+\sqrt{\alpha-\alpha_1}\, v_2(\alpha,\cdot)+\sqrt{\alpha-\alpha_2}\, v_3(\alpha,\cdot).
\end{equation}
\end{corollary}

\subsection{High order finite element method}

Based on the regularity results for the solution $w(\alpha,\cdot)$, we can employ the high order numerical method in \cite{Zhang2017e} to solve the problems in this paper. 

For simplicity, let's redefine the set $\Wast$. As all the points $\alpha\in\R$ such that there is an $n\in\Z$, $|\alpha+n\Lambda^*|=k$ lie $\Lambda^*$-periodically,  we can simply redefine $\Wast$ by
\begin{equation*}
\Wast:=(a_0,a_0+\Lambda^*],\quad\text{ where }a_0 \text{ is a point in  } \mathcal{S}_\Lambda.
\end{equation*}
Then we can also redefine $\mathcal{S}_\Lambda$ by the new $\Wast$ in the same way. There are two types of $\mathcal{S}_\Lambda$ depending on $k$, i.e., let $\underline{k}=\min\left\{|n\Lambda^*-k|:\,n\in\Z\right\}$ thus $\underline{k}\leq\Lambda^*/2$, then
\begin{enumerate}
\item if $\underline{k}=\frac{m\Lambda^*}{2}$ where $m=0,1$, then $\mathcal{S}_\Lambda=\left\{-\underline{k},\Lambda^*-\underline{k}\right\}$;
\item if $\underline{k}\neq\frac{m\Lambda^*}{2}$, then $\mathcal{S}_\Lambda=\left\{-\underline{k},\underline{k},\Lambda^*-\underline{k}\right\}$.
\end{enumerate}
For the high order method, we will change the variable $\alpha$. For the first case, let $\mathcal{S}_\Lambda=\{a_0,a_1\}$ and for the second case, let $\mathcal{S}_\Lambda=\{a_0,a_1,a_2\}$, then for each interval $[a_j,a_{j+1}]$ ($j=0$ for the first case, and $j=0,1$ for the second case), let $\alpha=g(t)$ such that $g$ satisfies the following properties
\begin{itemize}
\item $g$ is a monotonic increasing function in $[a_j,a_{j+1}]$;
\item $g(a_j)=a_j$, $g(a_{j+1})=a_{j+1}$, $g|_{[a_j,a_{j+1}]}\in C^\infty([a_j,a_{j+1}])$;
\item there is a $\epsilon>0$ such that
\begin{eqnarray*}
&& g(t)-a_j=C \delta^2(t-a_j),\quad t\in[a_j,a_j+\epsilon);\\
&& a_{j+1}-g(t)=C \delta^2(a_{j+1}-t),\quad t\in(a_{j+1}-\epsilon,a_{j+1}],
\end{eqnarray*}
where $\delta\in C^\infty(-2\epsilon,2\epsilon)$ such that either $\delta(t)=O(t^{m+1})$ for some $m\in\N$, or $\delta(t)=o(t^n)$ for any $n\in\N$ as $t\rightarrow 0^+$.
\end{itemize}

\begin{figure}[htb]
\centering
\begin{tabular}{c c}
\includegraphics[width=0.47\textwidth]{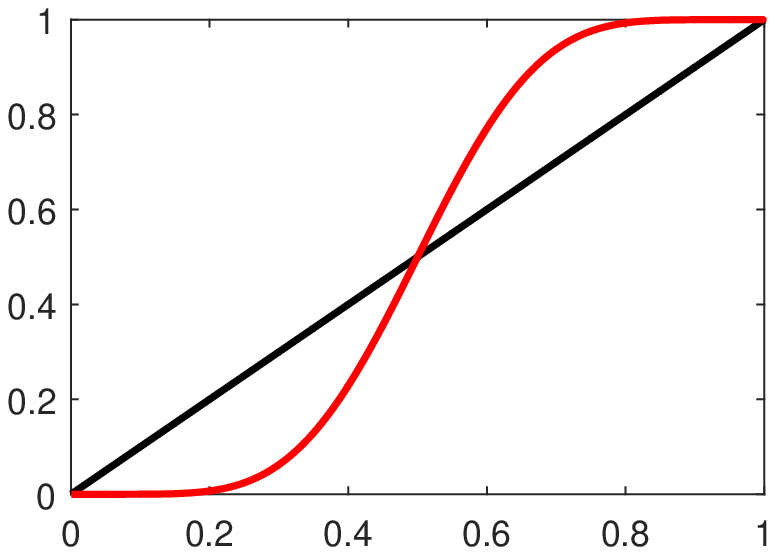} 
& \includegraphics[width=0.47\textwidth]{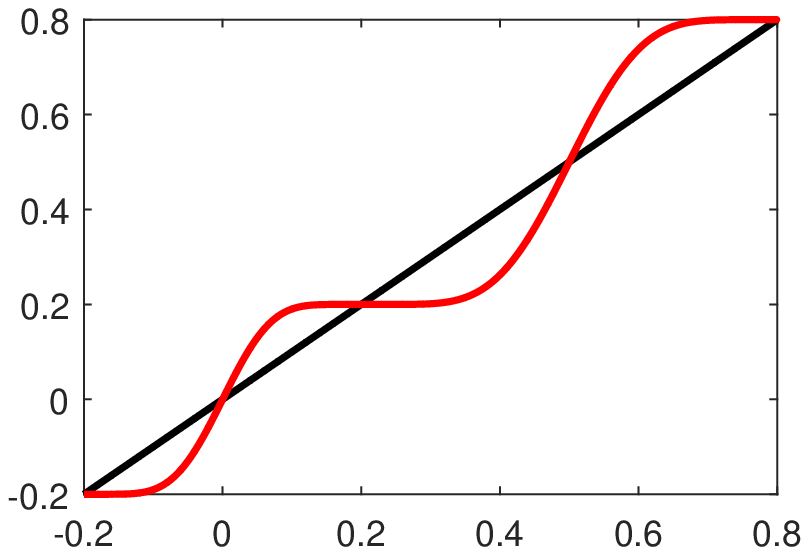}\\[-0cm]
(a) & (b)  
\end{tabular}%
\caption{(a)-(b): The two cases of of the locations of singularities and contours. (a) is the picture for $k=1$ and (b) is for $k=1.2$.}
\label{fig:contour}
\end{figure}

Let $v(t,\cdot):=w(g(t),\cdot)g'(t)$ and $\psi(t,\cdot):=\phi(g(t),\cdot)$, plunging the definitions in \eqref{eq:var_Bloch}, $v$ satisfies the following variational problem
\begin{equation}\label{eq:var_new}
\int_\Wast a_{g(t)}(v(t,\cdot),\psi(t,\cdot))\d t-F\left(\widetilde{\J}^{-1}_\Omega v,\widetilde{\J}^{-1}_\Omega[\psi(t,\cdot)g'(t)]\right)=F\left(u,\widetilde{\J}^{-1}_\Omega[\psi(t,\cdot)g'(t)]\right),
\end{equation}
where the new inverse Bloch transform $\widetilde{\J}^{-1}_\Omega$ is defined from the change of variables
\begin{equation*}
\begin{aligned}
(\J_\Omega^{-1} w)\left(x+\left(\begin{matrix}
\Lambda j\\0
\end{matrix}
\right)
\right)&=\left[\frac{\Lambda}{2\pi}\right]^{1/2}
\int_\Wast w(\alpha,x)e^{\i\alpha\Lambda j}\d \alpha\\
&=\left[\frac{\Lambda}{2\pi}\right]^{1/2}
\int_\Wast w(g(t),x)e^{\i g(t)\Lambda j}g'(t)\d t\\
&=\left[\frac{\Lambda}{2\pi}\right]^{1/2}
\int_\Wast v(t,x)e^{\i g(t)\Lambda j}\d t\\
&:=\left(\widetilde{\J}_\Omega^{-1} v\right)\left(x+\left(\begin{matrix}
\Lambda j\\0
\end{matrix}
\right)
\right)
\end{aligned}
\end{equation*}
However, from the definition of $F$ is on $\Omega^\Lambda_H$, which is the periodic cell where $j=0$, $(\J_\Omega^{-1}w)(x)=(\widetilde{\J}_\Omega^{-1}v)(x)$ for any $x\in\Omega^\Lambda$, and the variational problem is equivalent to 
\begin{equation}\label{eq:var_new1}
\int_\Wast a_{g(t)}(v(t,\cdot),\psi(t,\cdot))\d t-F\left({\J}^{-1}_\Omega v,{\J}^{-1}_\Omega[\psi(t,\cdot)g'(t)]\right)=F\left(u,{\J}^{-1}_\Omega[\psi(t,\cdot)g'(t)]\right),
\end{equation}

With the same finite element mesh and basic functions $\left\{\phi_M^{(\ell)}\right\}_{\ell=1}^M$ and $\left\{\psi^{(j)}_N\right\}_{j=1}^N$, we can define the new finite element space by
\begin{equation*}
 \widetilde{X}_{N,h}=\left\{v_{N,h}(t,x)=e^{-\i g(t) x_1}\sum_{j=1}^N\sum_{\ell=1}^M v^{(j,\ell)}_{N,h}\psi_N^{(j)}(t)\psi^{(\ell)}_M(x):\,v^{(j,\ell)}_{N,h}\in\C\right\},
\end{equation*}
thus the problem could be solved by Galerkin method. The high order numerical method could be concluded in the following algorithm.
\begin{algorithm}\label{alg:high_order}
High order numerical method for the scattering problems.
\begin{enumerate}
\item Find $u_h\in{V}^\alpha_h$ that solves the varaitional problem \eqref{eq:fem_var_per}.
\item Find $v_{N,h}\in \widetilde{X}_{N,h}$ that satisfies the following variational problem for all $\psi\in X ^0_{N,h}$
\begin{equation}
\int_\Wast a_{g(t)}(v_{N,h},\psi_{N,h})\d t-F\left({\J}^{-1}_{\Omega} v,{\J}^{-1}_{\Omega}[\psi_{N,h}g'(t)]\right)=F\left(u_h,{\J}^{-1}_{\Omega}[\psi_{N,h}g'(t)]\right),
\end{equation}
\item Let $v_{N,h}=\widetilde{\J}^{-1}_{\Omega}v_{N,h}$ and then approximate $u_T$ by $u_{T}^{N,h}=u_h+v_h$.
\end{enumerate}
\end{algorithm}

 From the arguments in previous sections and \cite{Zhang2017e}, the convergence result and error estimate could be concluded in the following theorem.

\begin{theorem}
Assume $\zeta_p,\zeta\in C^{2,1}(\R)$, then the linear system \eqref{eq:var_new} is uniquely solvable in $ X ^0_{N,h}$ for large enough $N$ and small enough $h>0$. Let $v_h$ be the exact solution of the variational problem for any $\psi\in C_0^\infty(\Wast\times\Omega^\Lambda_H)$
\begin{equation}
\int_\Wast a_{g(t)}(v_{h},\psi)\d t-F\left({\J}^{-1}_{\Omega} v_h,{\J}^{-1}_{\Omega}[\psi(t,\cdot) g'(t)]\right)=F\left(u_h,{\J}^{-1}_{\Omega}[\psi(t,\cdot)g'(t)]\right),
\end{equation}
then \\
\noindent
1) If $\delta(t)=O(t^{m+1})$ for some positive integer $m$, then the solution $v_{N,h}\in  X _{N,h}$ satisfies the error estimate between $v_{N,h}$ and $v_h$ or $v$
\begin{eqnarray*}
&&\|v_{N,h}-v_h\|_{L^2(\Wast;\widetilde{H}^\ell (\Omega^\Lambda_H))}\leq C h^{1-\ell}(N^{-2m+1/2}+h);\\
&&\|v_{N,h}-v\|_{L^2(\Wast;\widetilde{H}^1 (\Omega^\Lambda_H))}\leq C(N^{-2m+1/2}+h).
\end{eqnarray*}

\noindent
2) If $\delta(t)=o(t^{n+1})$ for any positive integer $n$, then the solution has the following error estimate
\begin{eqnarray*}
&&\|v_{N,h}-v_h\|_{L^2(\Wast;\widetilde{H}^\ell (\Omega^\Lambda_H))}\leq C h^{1-\ell}(N^{-2n+1/2}+h);\\
&&\|v_{N,h}-v\|_{L^2(\Wast;\widetilde{H}^1 (\Omega^\Lambda_H))}\leq C(N^{-2n+1/2}+h).
\end{eqnarray*}
\end{theorem}

\section{Further discussion for more generalized incident fields}

Although standard finite element method is available for $u^i\in H_r^1(\Omega_H)$ for any $r>0$, high order numerical method is only available when $u^i\in H_r^1(\Omega_H)$ for $r\in(1/2,1)$ (see \cite{Zhang2017e}).  In fact, the high order method could be extended to more generalized incident fields with the help of the technique introduced in this paper. The following assumption is necessary for the method.
\begin{assumption}
The incident field satisfies that, the scattering problems from periodic surfaces could be solved by known numerical methods. The following three cases of incident fields satisfy this condition.
\begin{enumerate}
\item Finite combination of quasi-periodic fields, e.g., finite combination of plain waves.
\item When $u^i\in H_r^1(\Omega^\Lambda_H)$ for any $r>0$.
\end{enumerate}
\end{assumption}

In this section, we only talk about the second case, for the first case is almost the same as that in previous sections. From \cite{Lechl2017}, when $r>0$, the problem could be solved numerically, by solving a large linear system of the size $M(N+1)\times M(N+1)$ has the following form (for the discretization, see \cite{Lechl2017})
\begin{equation}\label{eq:sys1}
\left(\begin{matrix}
A_1 & 0 & \dots & 0 & C_1\\
0 & A_2 & \dots & 0 & C_2\\
\vdots & \vdots & \vdots & \vdots &\vdots\\
0 & 0 & \dots & A_N & C_N\\
B_1 & B_2 & \dots & B_N &I
\end{matrix}
\right)\left(
\begin{matrix}
W_1\\W_2\\ \vdots \\W_N\\U
\end{matrix}
\right)=\left(
\begin{matrix}
F_1 \\ F_2 \\ \vdots \\ F_N\\ 0
\end{matrix}
\right).
\end{equation}
 The convergence rate with respect to $L^2$-norm is $h(N^{-r}+h)$. Which means that when $r$ is relatively small, a large $N$ is required for a small enough error. However, when the surface is not perturbed, the large system is decoupled into $N$ sub-systems of the size $M\times M$, which is much more easier for numerical schemes:
\begin{equation}\label{eq:sys2}
A_j W_j=F_j,\quad j=1,2,\dots,N.
\end{equation}
Then with the numerical solution $u_h$, we can apply either the standard finite element method, i.e., Algorithm \ref{alg}, or the high order numerical method, i.e., Algorithm \ref{alg:high_order} to solve the problem. 
 The convergence result and error estimation could be deduced in the same way, with the known result in \cite{Lechl2016a} and that in previous sections, and will not be discussed again.

\section{Numerical Results}

In this section, we will show some numerical results for both the standard Algorithm \ref{alg} and the high order Algorithm \ref{alg:high_order}. In each example, the incident field is a downward propagating plane wave defined by the parameter $k>0$ and $\alpha\in(-k,k)$:
\begin{equation*}
u^i_{k,\alpha}(x)=\exp\left(\i\alpha x_1-\i\sqrt{k^2-\alpha^2}x_2\right).
\end{equation*}
There are two $2\pi$-periodic (i.e., $\Lambda=2\pi$ and $\Lambda^*=1$) surfaces defined by
\begin{eqnarray*}
&& f_1(x_1)=1.9+\frac{\sin x_1}{3}-\frac{\cos 2x}{4};\\
&& f_2(x_1)=2-\frac{\cos x_1}{4}.
\end{eqnarray*}
And two local perturbations, defined by
\begin{eqnarray*}
&& p_1(x_1)= \exp\left(\frac{1}{x_1^2-1}\right)\sin\left[\pi(x_1+1)\right]\cdot\mathcal{X}_{(-1,1)}(x_1);\\
&& p_2(x_1)=\frac{1+\cos x}{4}\cdot\mathcal{X}_{(-\pi,\pi)}(x_1),
\end{eqnarray*}
where $\mathcal{X}_{(a,b)}$ is a smooth cut off function that equals to zero outside $(a,b)$. Then the two locally perturbed periodic functions are defined by $\zeta_j=f_j+p_j$, thus the two surfaces $\Gamma_1$ and $\Gamma_2$ are defined by
\begin{equation*}
\Gamma_j:=\left\{(x_1,\zeta_j(x_1)):\,x_1\in\R\right\}.
\end{equation*}

\begin{figure}[htb]
\centering
\begin{tabular}{c c}
\includegraphics[width=0.47\textwidth]{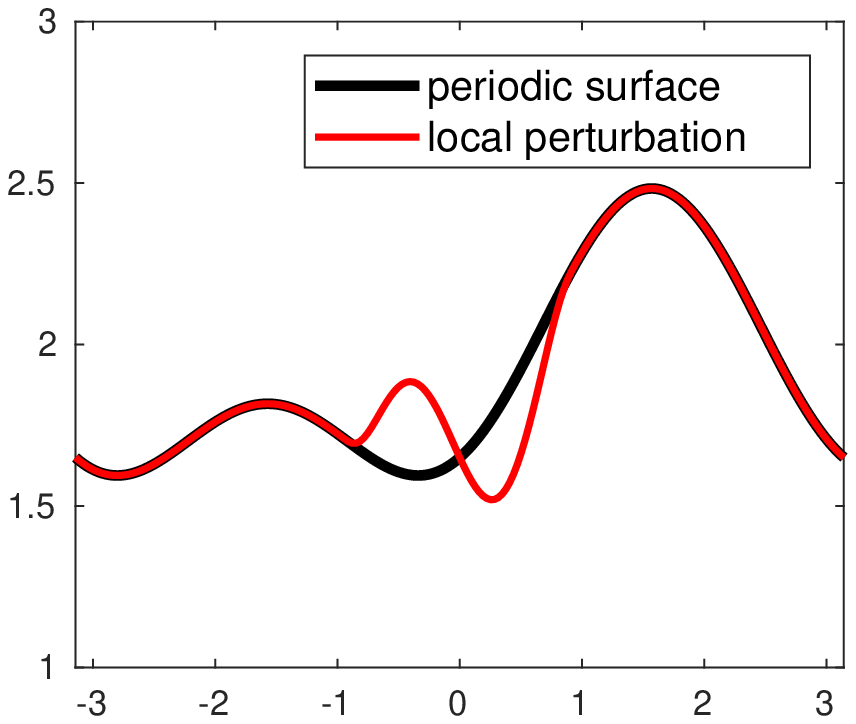} 
& \includegraphics[width=0.47\textwidth]{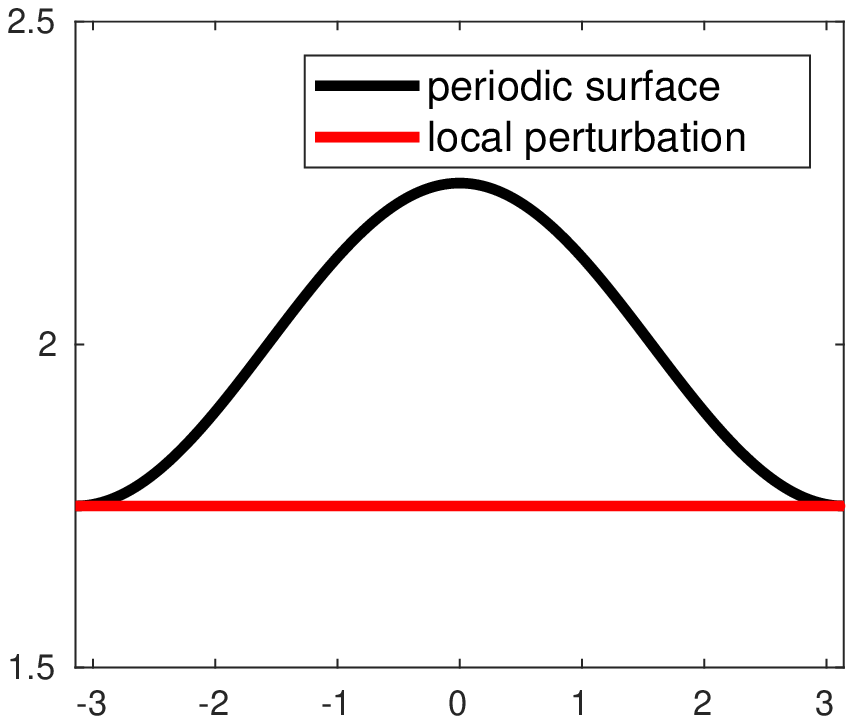}\\[-0cm]
(a) & (b)  
\end{tabular}%
\caption{Locally perturbed surfaces; (a): $\Gamma_1$, (b): $\Gamma_2$.}
\label{fig:surface}
\end{figure}

For each example, we just check the convergence rate of the second step, for the error estimation for the first step is very clear. The height of the straight line $\Gamma_H$ is always set to be $H=4$, and the parameter $H_0=3.9$. The data are collected one the line segment $\Gamma_H^\Lambda=(-\pi,\pi)\times\{4\}$.

\subsection{Standard method}

In this subsection, we will show four examples for the standard method. For numerical implementation see \cite{Lechl2017}. For each example, we set $h=0.16,0.08,0.04,0.02$ and $N=20,40,80,160,320$. The numerical solution when $h=0.02$ and $N=320$ is set as the "exact solution", then the $L^2$-relative error is defined by
\begin{equation*}
err_{N,h}\approx\frac{\|u_{N,h}-u_{320,0.02}\|_{L^2(\Gamma_H^\Lambda)}}{\|u_{320,0.02}\|_{L^2(\Gamma^\Lambda_H)}}.
\end{equation*}
 In Table \ref{surf1k1}-\ref{surf4k4}, the relative error between the numerical solution and the "exact solution" is shown. 

\begin{example}Examples for Algorithm \ref{alg}.
\begin{enumerate}
\item $k=1$, $\alpha=0.3$, $u^i$ is scattered by surface $\Gamma_1$;
\item $k=\sqrt{10}$, $\alpha=0.5$, $u^i$ is scattered by surface $\Gamma_1$;
\item $k=5$, $\alpha=-0.5$, $u^i$ is scattered by surface $\Gamma_2$;
\item $k=10$, $\alpha=\sqrt{2}$, $u^i$ is scattered by surface $\Gamma_2$.
\end{enumerate}
\end{example}

\begin{table}[htb]
\centering
\caption{Relative $L^2$-errors for Example 1 (surface $\Gamma_1$, $k=1$, $\alpha=0.3$).}\label{surf1k1}
\begin{tabular}
{|p{1.8cm}<{\centering}||p{2cm}<{\centering}|p{2cm}<{\centering}
 |p{2cm}<{\centering}|p{2cm}<{\centering}|p{2cm}<{\centering}|}
\hline
  & $h=0.16$ & $h=0.08$ & $h=0.04$ & $h=0.02$\\
\hline
\hline
$N=20$&$4.36$E$-03$&$3.64$E$-03$&$3.48$E$-03$&$3.45$E$-03$\\
\hline
$N=40$&$2.32$E$-03$&$1.38$E$-03$&$1.20$E$-03$&$1.16$E$-03$\\
\hline
$N=80$&$1.76$E$-03$&$6.44$E$-04$&$4.12$E$-04$&$3.73$E$-04$\\
\hline
$N=160$&$1.61$E$-03$&$4.41$E$-04$&$1.49$E$-04$&$9.67$E$-05$\\
\hline
$N=320$&$1.47$E$-03$&$3.93$E$-04$&$8.01$E$-05$&$--$\\
\hline
\end{tabular}
\end{table}

\begin{table}[htb]
\centering
\caption{Relative $L^2$-errors for Example 2 (surface $\Gamma_1$, $k=\sqrt{10}$, $\alpha=0.5$).}\label{surf2k2}
\begin{tabular}
{|p{1.8cm}<{\centering}||p{2cm}<{\centering}|p{2cm}<{\centering}
 |p{2cm}<{\centering}|p{2cm}<{\centering}|p{2cm}<{\centering}|}
\hline
  & $h=0.16$ & $h=0.08$ & $h=0.04$ & $h=0.02$\\
\hline
\hline
$N=20$&$8.39$E$-02$&$2.15$E$-02$&$6.36$E$-03$&$4.33$E$-03$\\
\hline
$N=40$&$8.33$E$-02$&$2.08$E$-02$&$4.67$E$-03$&$1.33$E$-03$\\
\hline
$N=80$&$8.32$E$-02$&$2.08$E$-02$&$4.68$E$-03$&$9.92$E$-04$\\
\hline
$N=160$&$8.31$E$-02$&$2.07$E$-02$&$4.54$E$-03$&$2.41$E$-04$\\
\hline
$N=320$&$8.31$E$-02$&$2.07$E$-02$&$4.52$E$-03$&$--$\\
\hline
\end{tabular}
\end{table}

\begin{table}[htb]
\centering
\caption{Relative $L^2$-errors for Example 3 (surface $\Gamma_2$, $k=5$, $\alpha=-0.5$).}\label{surf3k3}
\begin{tabular}
{|p{1.8cm}<{\centering}||p{2cm}<{\centering}|p{2cm}<{\centering}
 |p{2cm}<{\centering}|p{2cm}<{\centering}|p{2cm}<{\centering}|}
\hline
  & $h=0.16$ & $h=0.08$ & $h=0.04$ & $h=0.02$\\
\hline
\hline
$N=20$&$1.83$E$-01$&$4.74$E$-02$&$9.96$E$-03$&$1.30$E$-03$\\
\hline
$N=40$&$1.82$E$-01$&$4.71$E$-02$&$9.61$E$-03$&$4.01$E$-03$\\
\hline
$N=80$&$1.82$E$-01$&$4.70$E$-02$&$9.53$E$-03$&$1.23$E$-04$\\
\hline
$N=160$&$1.82$E$-01$&$4.70$E$-02$&$9.51$E$-03$&$3.12$E$-05$\\
\hline
$N=320$&$1.82$E$-01$&$4.70$E$-02$&$9.51$E$-03$&$--$\\
\hline
\end{tabular}
\end{table}

\begin{table}[htb]
\centering
\caption{Relative $L^2$-errors for Example 4 (surface $\Gamma_2$, $k=10$, $\alpha=\sqrt{2}$).}\label{surf4k4}
\begin{tabular}
{|p{1.8cm}<{\centering}||p{2cm}<{\centering}|p{2cm}<{\centering}
 |p{2cm}<{\centering}|p{2cm}<{\centering}|p{2cm}<{\centering}|}
\hline
  & $h=0.16$ & $h=0.08$ & $h=0.04$ & $h=0.02$\\
\hline
\hline
$N=20$&$8.68$E$-01$&$3.27$E$-01$&$7.19$E$-02$&$6.46$E$-04$\\
\hline
$N=40$&$8.68$E$-01$&$3.27$E$-01$&$7.17$E$-02$&$1.60$E$-04$\\
\hline
$N=80$&$8.68$E$-01$&$3.27$E$-01$&$7.16$E$-02$&$3.84$E$-05$\\
\hline
$N=160$&$8.68$E$-01$&$3.27$E$-01$&$7.16$E$-02$&$1.37$E$-05$\\
\hline
$N=320$&$8.68$E$-01$&$3.27$E$-01$&$7.16$E$-02$&$--$\\
\hline
\end{tabular}
\end{table}

In Table \ref{surf1k1}-\ref{surf4k4}, the error decreases as the $N$ increases or $h$ decreases.  When $k=1$, the error caused by $N$ is the dominant one, thus the decrease of error with respect to $h$ is comparatively small, see Table \ref{surf1k1}, while when $k=5,10$, the cases are the opposite, see Table \ref{surf2k2}-\ref{surf3k3}. When $k=\sqrt{10}$, we can see that the influence from $h$ and $N$ are mixed. When $h$ is small enough (see the results when $h=0.02$), the convergence rate with respect to $N$ is faster than $O(N^r)$ for any $r<1$ as expected. While when $N$ is large enough (see the results when $N=320$), the convergence rate with respect to $h$ is about $h^{-2}$. The numerical results present even better convergence result than expected.

\subsection{High order}

In this subsection, we will show four examples for the high order method. For numerical implementation see \cite{Zhang2017e}. For each example, we fix $h=0.03$ and set $N=4,8,16,32,64,128$, as only the convergence rate with respect to $N$ is interesting in this section. The numerical solution when $N=128$ is set as the "exact solution", then the $L^2$-relative error is defined by
\begin{equation*}
err_{N}\approx\frac{\|u_{N}-u_{128}\|_{L^2(\Gamma_H^\Lambda)}}{\|u_{128}\|_{L^2(\Gamma^\Lambda_H)}}.
\end{equation*}
In Table \ref{index3}-\ref{indexinfty}, the relative error between the numerical solution and the "exact solution" is shown. We choose different functions $g$ in each table. In Table \ref{index3}, the function $g_1$ is defined by
\begin{equation*}
g_1(t)=A_0+\frac{A_1-A_0}{\int_{A_0}^{A_1}(s-A_0)^3(s-A_1)^3\d s}\left[\int_{A_0}^t(s-A_0)^3(s-A_1)^3\d s\right],
\end{equation*}
thus $\delta_1(t)=O(t^2)$ as $t\rightarrow 0^+$; and in Table \ref{indexinfty}, the function $g_2$ is defined by
\begin{equation*}
g_2(t)=A_0+\frac{A_1-A_0}{\int_{A_0}^{A_1}\exp\left(\frac{1}{(s-A_0)(s-A_1)}\right)\d s}\left[\int_{A_0}^t\exp\left(\frac{1}{(s-A_0)(s-A_1)}\right)\d s\right],
\end{equation*}
then $\delta_2(t)=o(t^n)$ for any $n\in\N$ as $t\rightarrow 0^+$.

\begin{example}Examples for Algorithm \ref{alg:high_order}.
\begin{enumerate}
\item $k=1$, $\alpha=0.3$, $u^i$ is scattered by surface $\Gamma_1$;
\item $k=\sqrt{10}$, $\alpha=0.5$, $u^i$ is scattered by surface $\Gamma_1$;
\item $k=\sqrt{5}$, $\alpha=-0.5$, $u^i$ is scattered by surface $\Gamma_2$;
\item $k=10$, $\alpha=\sqrt{2}$, $u^i$ is scattered by surface $\Gamma_2$.
\end{enumerate}
\end{example}

\begin{table}[htb]
\centering
\caption{Relative $L^2$-errors ($g_13$).}\label{index3}
\begin{tabular}
{|p{1.8cm}<{\centering}||p{2cm}<{\centering}|p{2cm}<{\centering}
 |p{2cm}<{\centering}|p{2cm}<{\centering}|p{2cm}<{\centering}|}
\hline
  & Example 1 & Example 2 & Example 3 & Example 4\\
\hline
\hline
$N=4$&$1.90$E$-02$&$8.35$E$-03$&$9.07$E$-03$&$1.70$E$-01$\\
\hline
$N=8$&$8.04$E$-04$&$5.02$E$-04$&$6.08$E$-04$&$4.21$E$-03$\\
\hline
$N=16$&$5.01$E$-05$&$3.18$E$-05$&$3.86$E$-05$&$7.11$E$-05$\\
\hline
$N=32$&$3.12$E$-06$&$1.99$E$-06$&$2.41$E$-06$&$4.28$E$-06$\\
\hline
$N=64$&$1.83$E$-07$&$1.17$E$-07$&$1.42$E$-07$&$2.52$E$-07$\\
\hline
\end{tabular}
\end{table}

\begin{table}[htb]
\centering
\caption{Relative $L^2$-errors ($g_2$).}\label{indexinfty}
\begin{tabular}
{|p{1.8cm}<{\centering}||p{2cm}<{\centering}|p{2cm}<{\centering}
 |p{2cm}<{\centering}|p{2cm}<{\centering}|p{2cm}<{\centering}|}
\hline
  & Example 1 & Example 2 & Example 3 & Example 4\\
\hline
\hline
$N=4$&$5.39$E$-02$&$1.92$E$-02$&$9.68$E$-03$&$3.23$E$-01$\\
\hline
$N=8$&$1.14$E$-04$&$9.33$E$-05$&$9.51$E$-05$&$9.12$E$-03$\\
\hline
$N=16$&$1.15$E$-05$&$7.21$E$-06$&$8.73$E$-06$&$1.36$E$-04$\\
\hline
$N=32$&$1.45$E$-08$&$9.24$E$-09$&$1.12$E$-08$&$2.79$E$-08$\\
\hline
$N=64$&$1.17$E$-12$&$6.32$E$-13$&$8.05$E$-13$&$1.43$E$-12$\\
\hline
\end{tabular}
\end{table}

In each column of Table \ref{index3}, the convergence rate is around $O(N^{-4})$, which is a little faster than expected, i.e., $O(N^{-2\times 2+1/2})=O(N^{-3.5})$. While for results in Table \ref{indexinfty}, the convergence rate is much more faster as $N$ becomes larger. The results in the two tables are good illustrations
for the error estimations. Compared to the standard method, the high order method improves the convergence rate significantly.

\section*{Appendix I: The Floquet-Bloch transform}

Recall the definition of the Bloch transform on the one dimensional space $\R$. We can define the Bloch transform $\J_{\R}$ on the space $\R$.  Suppose $\phi\in C_0^\infty(\R)$, then the one dimensional Bloch transform $\J_{\R}$ of $\phi$ is defined by
\begin{equation*}
\left(\J_{\R}\phi\right)({\alpha},x_1)=\left[\frac{\Lambda}{2\pi}\right]^{1/2}\sum_{{j}\in\Z}\phi(x_1+\Lambda{ j})e^{-\i{\alpha}\Lambda{ j}},\quad{\alpha}\in\R,\, x_1\in\R.
\end{equation*} 
The Bloch transform $\J_{\R}$ is now well defined from $C_0^\infty(\R)$ to $C^\infty(\Wast\times\W)$. For any fixed $x_1\in\W$, the transformed field is $\Lambda^*$-periodic in ${\alpha}$; for any fixed ${\alpha}\in\Wast$, it is an ${\alpha}$-quasi-periodic in $x_1$ with period $\Lambda$, i.e., for any ${ j}\in\Z$,
\begin{eqnarray*}
&&\J_{\R}\phi({\alpha}+\Lambda^*{ j},{x_1})=\J_{\R}\phi({\alpha},x_1),\quad { \alpha}\in\R;\\
&&\J_{\R}\phi({\alpha},x_1+\Lambda{j})=e^{\i{\alpha} \Lambda{ j} }\J_{\R}\phi({\alpha},x_1),\quad x_1\in\R.
\end{eqnarray*}

To introduce the properties of the Bloch transform, we need to define some function spaces before that. For $s,r\in\R$, define the weighted function space
\begin{equation*}
H_r^s(\R):=\left\{\phi\in \mathcal{D}'(\R):\, (1+|x_1|^2)^{r/2}\phi(x_1)\in H^s(\R)\right\}. 
\end{equation*}
For $\ell\in\N$, $s\in\R$, define the function space 
\begin{equation*}
H^\ell(\Wast;H^s(\W)):=\left\{\psi\in \mathcal{D}'(\Wast\times\W):\,\sum_{m=0}^\ell\int_\Wast \left\|\partial^m_{\alpha}\psi({\alpha},\cdot)\right\|^2_{H^s(\Wast)}\d{\alpha}<\infty\right\}.
\end{equation*}
Extend $\ell\in\N$ to any $r>0$ by interpolation, and then extend to $r\in\R$ by duality arguments, then the definition is defined for any $\ell\in\R$. Define the space $H^r_0(\Wast;H^s_\alpha(\W))$ by the restriction of $H^r(\Wast;H^s(\W))$ such that the functions in this space are $\Lambda^*$-periodic in ${\alpha}$, and ${\alpha}$-quasi-periodic in $x_1$ with period $\Lambda$. Then the Bloch transform is extended into a bounded linear operator on a larger space.

\begin{theorem}
The Bloch transform $\J_{\R}$ extends to an isometric isomorphism between $H_r^s(\R)$ and $H_0^r(\Wast;H_\alpha^s(\W))$ for any $s,r\in\R$. Its inverse has the form of
\begin{equation*}
(\J^{-1}_{\R}\psi)(x_1+\Lambda{ j})=C_\Lambda\int_\Wast \psi({\alpha},x_1)e^{\i{\alpha} \Lambda{ j}}\d{\alpha},\quad x_1\in\W,\,{ j}\in\Z.
\end{equation*}
When $r=s=0$, the adjoint operator $\J^*_{\R}$ equals to the inverse $\J^{-1}_{\R}$.
\end{theorem}

We can extend the definition of $\J_{\R}$ to $n$-dimensional periodic domains similarly. Suppose $\Omega\subset\R^2$ is periodic in $x_1$-direction with period $\Lambda$, i.e., for any $x\in\Omega$, the point $x+(\Lambda { j},0)^\top\in\Omega,\,\forall{ j}\in\Z$. Moreover, assume there is a $L>0$ such that $\Omega\subset\R\times[-L,L]$. Define one periodic cell $\Omega^\Lambda=\Omega\cap\left[\W\times\R\right]$. For any $\phi\in C_0^\infty(\Omega)$, define the (partial) two dimensional Bloch transform $\J_\Omega$ of $\phi$ as
\begin{equation*}
\left(\J_\Omega\phi\right)({\alpha},{ x})=C_\Lambda\sum_{{j}\in\Z}\phi\left({ x}+\left(\begin{matrix}
\Lambda {j}\\0
\end{matrix}\right)\right)e^{-\i{\alpha}\Lambda{j}},\quad {\alpha}\in\R,\,{ x}\in\Omega^\Lambda.
\end{equation*}
\begin{remark}
Though the definition of the domain $\Omega$ and $\Omega^\Lambda$ might be different in the rest of this paper, we will use $\J_\Omega$ to be the $n$-dimensional (partial) Bloch transform on periodic domains with period $\Lambda$.
\end{remark}
We can also define the weighted Sobolev  space on the strip $\Omega$
\begin{equation*}
H_r^s(\Omega):=\left\{\phi\in \mathcal{D}'(\Omega):\,(1+|x|^2)^{r/2}\phi(x)\in H^s(\Omega)\right\}.
\end{equation*}
For $\ell\in\N$, $s\in\R$, we can also define
\begin{equation*}
H^\ell(\Wast;H^s(\Omega^\Lambda)):=\left\{\psi\in\mathcal{D}'(\Wast\times\Omega^\Lambda):\,\sum_{m=0}^\ell\int_\Wast\left\|\partial^m_{\alpha}\psi({\alpha},\cdot)\right\|\d{\alpha}<\infty\right\},
\end{equation*}
and extend to any $r\in\R$ by interpolation and duality  similarly. The space $H_0^r(\Wast;H_\alpha^s(\Omega^\Lambda))$ could be defined in the same way. We can have the following properties for the $n$-dimensional (partial) Bloch transform $\J_\Omega$.

\begin{theorem}\label{th:Bloch_property}
The Bloch transform $\J_\Omega$ extends to an isometric isomorphism between $H_r^s(\Omega)$ and $H_0^r(\Wast;H_\alpha^s(\Omega^\Lambda))$ for any $s,r\in\R$. Its inverse has the form of
\begin{equation*}
(\J^{-1}_\Omega\psi)\left({ x}+\left(\begin{matrix}
\Lambda { j}\\0
\end{matrix}\right)\right)=\left[\frac{\Lambda}{2\pi}\right]^{1/2}\int_\Wast \psi({\alpha},{ x})e^{\i{\alpha}\Lambda{ j}}\d{\alpha},\quad x_1\in\Omega^\Lambda,\,{j}\in\Z.
\end{equation*}
When $r=s=0$, the adjoint operator $\J^*_\Omega$ equals to the inverse $\J^{-1}_\Omega$. 
\end{theorem}

Another important property of the Bloch transform is the commutes with partial derivatives, see \cite{Lechl2016}. If $u\in H_r^n(\Omega)$ for some $n\in\N$, then for any ${\bm\gamma}=(\gamma_1,\gamma_2)\in\N^2$ with $|\gamma|=|\gamma_1|+|\gamma_2|\leq N$,
\begin{equation*}
\partial^{\bm\gamma}_{\bm x} \left(\J_\Omega u\right)({\bm\alpha},{\bm x})=\J_\Omega[\partial^{\bm\gamma} u]({\bm\alpha},{\bm x}).
\end{equation*}

\begin{remark}\label{rem:Four}
There is an alternative definition for the space $H_0^r(\Wast;X_{\alpha})$, where $X_{\alpha}$ is a family of Hilbert spaces that are ${\alpha}$-quasi-periodic in $x_1$. Let 
\begin{equation*}
\phi_{\Lambda^*}^{({ j})}({\alpha})=\frac{1}{\sqrt{|\Lambda^*|}}e^{\i{\alpha}\Lambda { j}},\,{ j}\in\Z
\end{equation*}
be a complete orthonormal system in $L^2(\Wast)$, then any function $\psi\in \mathcal{D}'(\Wast\times\Omega^\Lambda)$ has a Fourier series
\begin{equation*}
\psi({\alpha},{ x})=\frac{1}{\sqrt{\Lambda^*}}\sum_{{\ell}\in\Z}\hat{\psi}_{\Lambda^*}({\ell},{ x})e^{\i{\alpha}\Lambda{\ell}},
\end{equation*}
where $\hat{\psi}_{\Lambda^*}({\ell},{ x})=<\psi(\cdot,{ x}),\phi_{\Lambda^*}^{({\ell})}>_{L^2(\Wast)}$. Then the squared norm of any $\psi\in H_0^r(\Wast;X_{\alpha})$ equals to
\begin{equation*}
\|\psi\|^2_{H_0^r(\Wast;X_{\alpha})}=\sum_{{\ell}\in\Z}(1+|{\ell}|^2)^r\left\|\hat{\psi}_{\Lambda^*}({\ell},\cdot)\right\|^2_{X_{\alpha}}.
\end{equation*}
\end{remark}

\bibliographystyle{alpha}
\bibliography{Split.bib} 

\newcommand{\etalchar}[1]{$^{#1}$}
\providecommand{\noopsort}[1]{}
\begin{thebibliography}{MACK00}

\bibitem[BS94]{Brenn1994}
S.~C. Brenner and L.~R. Scott.
\newblock {\em The Mathematical Theory of Finite Element Methods}.
\newblock Springer, New York, 1994.

\bibitem[CE10]{Chand2010}
S.~N. {Chandler-Wilde} and J.~Elschner.
\newblock Variational approach in weighted {S}obolev spaces to scattering by
  unbounded rough surfaces.
\newblock {\em SIAM. J. Math. Anal.}, 42:2554--2580, 2010.

\bibitem[Coa12]{Coatl2012}
J.~Coatl{\'e}ven.
\newblock {Helmholtz equation in periodic media with a line defect}.
\newblock {\em {J. Comp. Phys.}}, 231:1675--1704, 2012.

\bibitem[CWR96a]{Chand1996}
S.~N. Chandler-Wilde and C.R. Ross.
\newblock Scattering by rough surfaces: the {D}irichlet problem for the
  {H}elmholtz equation in a non-locally perturbed half-plane.
\newblock {\em Math. Meth. Appl. Sci.}, 19:959--976, 1996.

\bibitem[CWR96b]{Chand1996a}
S.~N. Chandler-Wilde and C.R. Ross.
\newblock Uniqueness results for direct and inverse scattering by infinite
  surfaces in a lossy medium.
\newblock {\em Inverse Problems}, 11:1063--1067, 1996.

\bibitem[CWZ98]{Chand1998}
S.~N. Chandler-Wilde and B.~Zhang.
\newblock A uniqueness result for scattering by infinite dimensional rough
  surfaces.
\newblock {\em SIAM J. Appl. Math.}, 58:1774--1790, 1998.

\bibitem[HLQZ15]{Hu2015}
G.~Hu, X.~Liu, F.~Qu, and B.~Zhang.
\newblock Variational approach to scattering by unbounded rough surfaces with
  {N}eumann and generalized impedance boundary conditions.
\newblock {\em Commun. Math. Sci.}, 13(2):511--537, 2015.

\bibitem[HNPX11]{George2011}
G.~C. Hsiao, N.~Nigam, J.~E. Pasciak, and L.~Xu.
\newblock Error analysis of the {DtN-FEM} for the scattering problem in
  acoustics via {F}ourier analysis.
\newblock {\em J. Comp. Appl. Math.}, 235(17):4949--4965, 2011.

\bibitem[Kir93]{Kirsc1993}
A.~Kirsch.
\newblock Diffraction by periodic structures.
\newblock In L.~P{\"a}varinta and E.~Somersalo, editors, {\em Proc. Lapland
  Conf. on Inverse Problems}, pages 87--102. Springer, 1993.

\bibitem[Kir95]{Kirsc1995a}
A.~Kirsch.
\newblock An inverse scattering problem for periodic structures.
\newblock In R.~Kress E.~Martensen {R.E. Kleinman}, editor, {\em Methoden und
  Verfahren der mathematischen Physik}, pages 75--93. Peter Lang, 1995.

\bibitem[Lec17]{Lechl2016}
A.~Lechleiter.
\newblock The {F}loquet-{B}loch transform and scattering from locally perturbed
  periodic surfaces.
\newblock {\em J. Math. Anal. Appl.}, 446(1):605--627, 2017.

\bibitem[LN15]{Lechl2015e}
A.~Lechleiter and D.-L. Nguyen.
\newblock {Scattering of {H}erglotz waves from periodic structures and mapping
  properties of the {B}loch transform}.
\newblock {\em {Proc. Roy. Soc. Edinburgh Sect. A}}, 231:1283--1311, 2015.

\bibitem[LZ16]{Lechl2016b}
A.~Lechleiter and R.~Zhang.
\newblock Non-periodic acoustic and electromagnetic scattering from periodic
  structures in 3d.
\newblock {\em Accepted by Comput. Math. Appl.}, 2016.

\bibitem[LZ17a]{Lechl2016a}
A.~Lechleiter and R.~Zhang.
\newblock A convergent numerical scheme for scattering of aperiodic waves from
  periodic surfaces based on the {F}loquet-{B}loch transform.
\newblock {\em SIAM J. Numer. Anal}, 55(2):713--736, 2017.

\bibitem[LZ17b]{Lechl2017}
A.~Lechleiter and R.~Zhang.
\newblock A {F}loquet-{B}loch transform based numerical method for scattering
  from locally perturbed periodic surfaces.
\newblock {\em SIAM J. Sci. Comput.}, 39(5):B819--B839, 2017.

\bibitem[MACK00]{Meier2000}
A.~Meier, T.~Arens, S.~N. {Chandler-Wilde}, and A.~Kirsch.
\newblock A {N}ystr{\"o}m method for a class of integral equations on the real
  line with applications to scattering by diffraction gratings and rough
  surfaces.
\newblock {\em J. Int. Equ. Appl.}, 12:281--321, 2000.

\bibitem[MB79]{Munk1979}
B.~A. Munk and G.~A. Burrell.
\newblock Plane-wave expansion for arrays of arbitrarily oriented piecewise
  linear elements and its application in determing the impedance of a single
  linear antenna in a lossy half-plane.
\newblock {\em IEEE Trans. Antenn. Propag.}, 27:331--343, 1979.

\bibitem[VBB{\etalchar{+}}08]{Valer2008}
G.~Valerio, P.~Baccarelli, P.~Burghignoli, A.~Galli, R.~Rodriguez-Berral, and
  F.~Mesa.
\newblock Analysis of periodic shielded microstrip lines excited by nonperiodic
  sources through the array scanning method.
\newblock {\em Radio Sci.}, 43:RS1009, 2008.

\bibitem[Zha17a]{Zhang2017e}
R.~Zhang.
\newblock A high order numerical method for scattering from locally perturbed
  periodic surfaces.
\newblock {\em Preprint}, 2017.

\bibitem[Zha17b]{Zhang2017d}
R.~Zhang.
\newblock The study of the bloch transform of the fields scattered by locally
  perturbed periodic surfaces.
\newblock {\em Preprint}, 2017.

\end{thebibliography}

\end{document}